\documentclass[11pt]{amsart}
\usepackage[a4paper,margin=1in]{geometry}
\usepackage[T1]{fontenc}
\usepackage{lmodern}
\usepackage{microtype}
\usepackage{amsmath,amssymb,amsthm,mathtools}
\usepackage{enumitem}
\usepackage{hyperref}
\usepackage[nameinlink,capitalise,noabbrev]{cleveref}
\hypersetup{
 colorlinks=true,
 allcolors=blue,
 pdftitle={Bose--Einstein Condensation without an Initial Low-Energy Concentration Assumption},
 pdfauthor={Siwei Luo and Jian-Guo Liu}
}
\numberwithin{equation}{section}
\newtheorem{theorem}{Theorem}[section]
\newtheorem{proposition}[theorem]{Proposition}
\newtheorem{lemma}[theorem]{Lemma}

\newtheorem{assumption}[theorem]{Assumption}
\theoremstyle{definition}
\newtheorem{definition}[theorem]{Definition}
\theoremstyle{remark}

\crefname{assumption}{Assumption}{Assumptions}
\crefname{remark}{Remark}{Remarks}
\Crefname{assumption}{Assumption}{Assumptions}
\Crefname{remark}{Remark}{Remarks}
\newcommand{\norm}[1]{\left\|#1\right\|}
\newcommand{\abs}[1]{\left|#1\right|}
\title[BEC without an Initial Low-Energy Concentration Assumption]
{Bose--Einstein Condensation without an Initial Low-Energy Concentration Assumption}
\author{Siwei Luo}
\address{School of the Gifted Young, University of Science and Technology of
China, Hefei 230026, China}
\email{luosw@mail.ustc.edu.cn}
\author{Jian-Guo Liu}
\address{Departments of Mathematics and Physics, Duke University, Durham,
North Carolina 27708, USA}
\email{jian-guo.liu@duke.edu}
\date{September 2026}
\thanks{Corresponding author: Siwei Luo
(\href{mailto:luosw@mail.ustc.edu.cn}
{luosw@mail.ustc.edu.cn}).}
\thanks{ORCID: Siwei Luo,
\href{https://orcid.org/0009-0001-1764-0843}
{0009-0001-1764-0843};
Jian-Guo Liu,
\href{https://orcid.org/0000-0002-9911-4045}
{0000-0002-9911-4045}.}
\begin{document}
\begin{abstract}
We prove that semi-strong convergence to a
Bose--Einstein equilibrium below the critical temperature implies finite-time Bose--Einstein condensation (BEC) for conservative isotropic measure solutions of the spatially homogeneous quantum Boltzmann equation with low-momentum scattering degeneracy exponent $0\leq\eta<1$. For every admissible initial measure with $\overline T/\overline T_c<1$, there exists a conservative solution that has a positive zero-energy atom after a finite time without any assumption of initial concentration near zero energy. We also show that it persists and converges to the equilibrium condensate mass, while the full solution converges strongly. The proof introduces a method to obtain condensation from relaxation estimate to Bose--Einstein equilibrium. The relaxation estimate first gives a fixed amount of mass below a small energy level. We then use collision estimates at smaller and smaller energy scales and add their effects in a weighted sum functional with a fixed upper bound. If no atom forms at zero energy, enough mass remains in the positive-energy shells to force this sum to exceed its upper bound in finite time, which leads to condensation in finite time.
\end{abstract}
\maketitle
\section{Introduction}\label{sec:introduction}
\subsection{Model and condensation problem}\label{sec:intro-model}
Bose--Einstein condensation (BEC) is a quantum collective phenomenon: at sufficiently low temperature, a macroscopic fraction of bosonic particles may accumulate in the lowest-energy state. The phenomenon originates from Bose's quantum statistics for indistinguishable particles and Einstein's extension of this statistics to an ideal gas of material particles, which led to the prediction of the condensation transition \cite{Bose1924,Einstein1924,Einstein1925}. A mesoscopic kinetic model for Bose--Einstein particles is the spatially homogeneous quantum Boltzmann equation \begin{equation}\label{eq:BN} \partial_t f(v,t) = \int_{\mathbb R^3}\int_{\mathbb S^2} B(v-v_*,\omega) \bigl[ f'f_*'(1+f)(1+f_*) - ff_*(1+f')(1+f_*') \bigr] \,d\omega\,dv_*, \end{equation} where $f=f(v,t)\geq0$ is the particle number density at time $t$ and velocity $v$, with \[ f_*=f(v_*,t),\qquad f'=f(v',t),\qquad f_*'=f(v_*',t), \] and the post-collisional velocities are \[ v'=v-\bigl((v-v_*)\cdot\omega\bigr)\omega, \qquad v_*'=v_*+\bigl((v-v_*)\cdot\omega\bigr)\omega, \qquad \omega\in\mathbb S^2. \] The quantum Boltzmann equation for Bose particles goes back to Nordheim \cite{Nordheim1928} and Uehling--Uhlenbeck \cite{UehlingUhlenbeck1933}. Its derivation from microscopic quantum dynamics and its mathematical theory have been studied in a variety of settings; see, among others, \cite{BenedettoEtAl2005,Lu2004,Lu2005,BriantEinav2016,LiLu2019}. A characteristic feature of \eqref{eq:BN}, different from the classical Boltzmann equation, is the appearance of the Bose-enhancement factors $1+f$. According to Bose statistics, the presence of particles in a given one-particle state increases the probability that further particles enter the same state. Thus, in a binary collision, transitions into the states $v$ and $v_*$ are enhanced by the factors $(1+f)(1+f_*)$, while transitions into $v'$ and $v_*'$ are enhanced by $(1+f')(1+f_*')$. This statistical amplification of already occupied states is the basic quantum mechanism through which the collision dynamics can drive particles toward highly occupied low-energy states.
For isotropic solutions, we can use the kinetic-energy variable \( x={|v|^2}/{2}. \) In the corresponding measure-valued formulation \(F_t(dx)\), Bose--Einstein condensation is represented by the appearance of a nonzero atom at the zero-energy state $x=0$. Below the critical temperature \(\overline T_c\), the Bose--Einstein equilibrium contains such a zero-energy atom, and this leads to a fundamental question in dynamics: \textit{given an initial state whose mass and energy correspond to a temperature below the critical temperature, can the collision dynamics create a condensate at $x=0$, and can this happen in finite time?}  For the hard-sphere Nordheim equation, finite-time condensation was established by Lu under assumptions on the initial concentration near zero energy \cite{Lu2013,Lu2014}.  Escobedo and Vel\'azquez then established condensation results for general measure-valued initial data under the only assumption \(\overline T/\overline T_c<1\) \cite{EscobedoVelazquez2014,EscobedoVelazquez2015}.
At the same time, the interaction potential has a significant impact on both the occurrence of Bose--Einstein condensation and its long-time convergence. To make the role of the particle interaction more explicit, we now specify the collision kernel. In the weak-coupling regime, microscopic derivations of the quantum Boltzmann equation lead, after a normalization of physical constants, to a collision kernel of the form \cite{BenedettoEtAl2005,ErdosSalmhoferYau2004} \begin{equation}\label{eq:B} B(v-v_*,\omega) = \frac{|(v-v_*)\cdot\omega|}{(4\pi)^2} \Phi\bigl(|v-v'|,|v-v_*'|\bigr), \end{equation}
where, for a real radially symmetric two-body interaction potential $\phi=\phi(|x|)$, \begin{equation}\label{eq:Phi-potential} \Phi(r,\rho) = \bigl(\widehat\phi(r)+\widehat\phi(\rho)\bigr)^2, \qquad r,\rho\geq0. \end{equation}
Here $\widehat\phi$ denotes the radial Fourier transform of $\phi$. The two terms in \eqref{eq:Phi-potential} correspond to the direct and exchange scattering amplitudes for identical bosons, and thus the low-momentum behavior of $\widehat\phi$ directly determines the strength of collisions near zero energy. The hard-sphere model is obtained from the constant choice $\widehat\phi\equiv1/2$, for which $\Phi\equiv1$ and \[ B(v-v_*,\omega) = \frac{1}{(4\pi)^2}|(v-v_*)\cdot\omega|. \]
Cai and Lu \cite{CaiLu2019} studied a class of such
kernels for which $\widehat\phi$ is continuous and nondecreasing and, for
some $b_0>0$ and $0\leq\eta<1/4$, satisfies
\[
 b_0\frac{r^\eta}{1+r^\eta}
 \leq \widehat\phi(r)\leq\frac12,
 \qquad r\geq0.
\]
For $0\leq\eta<1/4$, they proved strong convergence to the
Bose--Einstein equilibrium under the standard low-temperature condition
$\overline T/\overline T_c<1$, without imposing local assumptions on the
initial distribution near zero energy.
Their recent work \cite{CaiLu2026} isolates the properties of the scattering cross section and treats a broader class of kernels in which the scattering factor $\Phi$ need not be explicitly generated by an interaction potential. In particular, they quantify the possible degeneracy of $\Phi$ at low momentum transfer through an exponent $\eta$: for $0\leq\eta<1$, their condensation theory assumes the lower bound \[ \Phi(r,\rho)\geq b_0\min\{1,(r^2+\rho^2)^\eta\}, \] with the symmetry and monotonicity properties of $\Phi$. Thus larger values of $\eta$ allow a stronger weakening of the collision interaction as $(r,\rho)\to(0,0)$. Under this assumption, they obtained convergence to Bose--Einstein condensation and strong convergence for $0\leq\eta<1$ provided the initial data satisfy an additional quantitative low-energy concentration condition. Indeed, with \(0<\alpha<1-\eta,\beta=(1-\eta-\alpha)/2\), and \(p=3/2+\alpha\), they suppose that the initial value \(F_0(dx)\) in energy variable satisfies
\begin{align}\label{eq:ultra-low}
    \inf_{0<\delta\leq 2E/N}\dfrac{F_0([0,\delta])}{\delta^\alpha} &\geq \dfrac{2^\alpha 18}{b_0}\left(\dfrac{4\sqrt2}{3-p}\right)^2\left(\dfrac p \beta\right)^{2p}\left(\dfrac {4E}N\right)^{1/2+2\beta}\max\left\{\dfrac{4E}{N},1\right\}^\eta,\\
N &=\int_{[0,\infty)} dF_0(x),\qquad E=\int_{[0,\infty)}x\,dF_0(x).\notag
\end{align}
This in particular forces the ultra-low-temperature bound \begin{equation}\label{eq:Cai-Lu-temp}
    \dfrac{\overline T}{\overline T_c}<\frac{1}{600}b_0^{2/3}(1-\eta)^2.
\end{equation}
In the opposite direction, under their corresponding upper-degeneracy assumption with $\eta\geq1$, they showed that an initially absent zero-energy atom cannot be dynamically created. Motivated by this threshold behavior, and following their formulation in the range $0\leq\eta<1$, we impose the following assumption.
\begin{assumption}[Assumptions on \(\Phi\)]\label{ass:CL} The scattering factor $\Phi\in C([0,\infty)^2)$ is nonnegative, symmetric, and nondecreasing in each coordinate. For some $b_0\in(0,1)$ and $\eta\in[0,1)$, \begin{equation}\label{eq:Phi-power} b_0\min\{1,(r^2+\rho^2)^\eta\} \leq \Phi(r,\rho) \leq1, \qquad r,\rho\geq0. \end{equation} \end{assumption}
Our main result removes the additional initial low-energy concentration
condition \eqref{eq:ultra-low} for the full range $0\leq\eta<1$.
For $0\leq\eta<1/4$, this conclusion was already obtained in
\cite{CaiLu2019}; thus the new range is $1/4\leq\eta<1$. Consequently, the admissible temperature range is extended from the
ultra-low regime \eqref{eq:Cai-Lu-temp} to the full subcritical range
$\overline T/\overline T_c<1$.
More precisely, once the mass and energy correspond to a Bose--Einstein
equilibrium below the critical temperature, semi-strong relaxation to
that equilibrium, proved in \cite{CaiLu2026}, forces a nonzero zero-energy atom to be present after a finite time.  The condensate mass then converges to its equilibrium
value, and the full solution converges strongly to the Bose--Einstein
equilibrium.
\subsection{Basic concepts and notation}
We work throughout the paper with the isotropic formulation in the energy variable \( x={|v|^2}/{2}\in[0,\infty)\) in 3D, and denote by $F_t(dx)$ the corresponding nonnegative energy measure at time $t$. In this formulation, Bose--Einstein condensation is represented by the appearance of an atom at zero energy. Denote the condensate mass by \( N_c(t):=F_t(\{0\}), \) and write
\begin{equation}\label{eq:measure-atom-decomposition} F_t=N_c(t)\delta_0+F_t^{\circ}, \qquad F_t^{\circ}(\{0\})=0, \end{equation}
where $F_t^{\circ}$ denotes the non-condensed part of the measure. Whenever the non-condensed part is absolutely continuous with respect to $\sqrt{x}\,dx$, we write \begin{equation}\label{eq:measure-density-decomposition} dF_t(x) = g_t(x)\sqrt{x}\,dx+N_c(t)\delta_0(dx). \end{equation}
In particular, if the isotropic velocity distribution has the form \(f(v,t)=g_t({|v|^2}/{2})\), its absolutely continuous energy measure is $g_t(x)\sqrt{x}\,dx$ after dividing out the radial factor $4\pi\sqrt2$. For a nonnegative Borel measure $F$ on $[0,\infty)$ with finite mass and finite first moment, we denote its particle number and energy by
\begin{equation}\label{eq:mass-energy} N(F):=\int_{[0,\infty)}dF(x), \qquad E(F):=\int_{[0,\infty)}x\,dF(x). \end{equation}
For a conservative solution $(F_t)_{t\geq0}$, both quantities are independent of time, and we simply write
\[ N:=N(F_t), \qquad E:=E(F_t). \]
We consider the space of measures
\begin{equation}\label{eq:B1-space} \mathcal B_1^+([0,\infty)) := \left\{ F:\ F\ \text{is a nonnegative Borel measure on }[0,\infty), \quad \int_{[0,\infty)}(1+x)\,dF(x)<\infty \right\}. \end{equation} Thus $F\in\mathcal B_1^+([0,\infty))$ precisely when both $N(F)$ and $E(F)$ are finite.
For particle number $N>0$ and energy $E>0$, we denote by $\overline T$ the temperature determined by particle number and energy, and by $\overline T_c$ the corresponding critical temperature for Bose--Einstein condensation. In the normalization adopted here and in \cite{CaiLu2026}, \begin{equation}\label{eq:temperature-ratio} \frac{\overline T}{\overline T_c} = \frac{(2\pi)^{1/3}[\zeta(3/2)]^{5/3}} {3\zeta(5/2)} \frac{E}{N^{5/3}}, \end{equation}
where \(\displaystyle \zeta(s)=\sum_{n\ge 1}n^{-s}\) is the Riemann zeta function. Accordingly, \( {\overline T}/{\overline T_c}<1 \) is the subcritical-temperature regime for Bose--Einstein condensation. Since $N$ and $E$ are conserved, the ratio $\overline T/\overline T_c$ is fixed along the evolution.
We denote by $F_{\rm be}$ the Bose--Einstein equilibrium having the same particle number $N$ and energy $E$ as the solution. Its measure representation is \begin{equation}\label{eq:Fbe} dF_{\rm be}(x) = g_{\rm be}(x)\sqrt{x}\,dx + N_c\,\delta_0(dx), \qquad g_{\rm be}(x) = \frac{1}{A e^{x/\kappa}-1}, \end{equation}
where $A\geq1$ and $\kappa>0$ are determined by $N$ and $E$ \cite{Lu2004,CaiLu2026}. When $\overline T/\overline T_c\leq1$, one has $A=1$, and the equilibrium condensate mass is
\begin{equation}\label{eq:Nc} N_c:=F_{\rm be}(\{0\}) = \left( 1- \left(\frac{\overline T}{\overline T_c}\right)^{3/5} \right)_+N, \qquad a_+:=\max\{a,0\}. \end{equation}
In particular,
\begin{equation}\label{eq:Nc-positive} N_c>0 \iff {\overline T}/{\overline T_c}<1. \end{equation}
 For a signed Borel measure $\mu$ on $[0,\infty)$, we define \begin{equation}\label{eq:weighted-norms} \|\mu\|_1^\circ := \int_{[0,\infty)}x\,d|\mu|(x), \qquad \|\mu\|_1 := \int_{[0,\infty)}(1+x)\,d|\mu|(x). \end{equation} The quantity $\|\cdot\|_1^\circ$ is a seminorm without information for  mass concentration at $x=0$. We refer to convergence with respect to $\|\cdot\|_1^\circ$ as \emph{semi-strong convergence}, and to convergence with respect to $\|\cdot\|_1$ as \emph{strong convergence}.
 We use the following conventions throughout the paper.  For
nonnegative quantities \(X\) and \(Y\), the notation
\(X\lesssim_{\mathbf p}Y\) means that \(X\leq C_{\mathbf p}Y\) for a
constant \(C_{\mathbf p}>0\) depending only on the parameters listed in
\(\mathbf p\), while \(X\asymp_{\mathbf p}Y\) means that both
\(X\lesssim_{\mathbf p}Y\) and \(Y\lesssim_{\mathbf p}X\) hold.  When
the subscript is omitted, the comparison constant is independent of all
quantities varying in the estimate.  Similarly,
\(O_{\mathbf p}(1)\) denotes a quantity bounded in absolute value by a
constant depending on \(\mathbf p\).  Such constants may change
from line to line.  We write \(a\wedge b=\min\{a,b\}\),
\(\mathbf 1_A\) for the indicator of a set \(A\),
\(\operatorname{supp}\varphi\) for the support of \(\varphi\), and
\[
 \operatorname{Lip}(\varphi)
 =\sup_{x\ne y}\frac{|\varphi(x)-\varphi(y)|}{|x-y|}
\]
for its Lipschitz seminorm.  Finally, \(\delta_0\) denotes the Dirac
measure at zero, \(|\mu|\) the total-variation measure of a signed
measure \(\mu\), and \(\mathrm i=\sqrt{-1}\).
\subsection{Main results}\label{sec:intro-results} After introducing the notations, we can now state our main results rigorously. Our first theorem shows that once a conservative isotropic measure solution relaxes semi-strongly toward a Bose--Einstein equilibrium while \(\overline T/\overline T_c<1\), a zero-energy atom is present after some finite time. Moreover, the condensate mass converges to its equilibrium value, and the semi-strong convergence is upgraded to strong convergence.
\begin{theorem}[Relaxation to condensation] \label{thm:conditional-main} Assume \cref{ass:CL}, and let $(F_t)_{t\geq0}$ be a conservative isotropic measure solution with particle number $N>0$ and energy $E>0$. Let $F_{\rm be}$ be the Bose--Einstein equilibrium with the same particle number and energy. Suppose that \({\overline T}/{\overline T_c}<1, \)
or equivalently, $N_c=F_{\rm be}(\{0\})>0$, and that \(\|F_t-F_{\rm be}\|_1^\circ\to 0\) as \(t\to\infty\). Then there exists a finite time $T\geq0$ such that \begin{equation}\label{eq:atom-positive-main} N_c(T)=F_T(\{0\})>0. \end{equation} Furthermore, the condensate persists for all later times: \begin{equation}\label{eq:atom-persistence-main} N_c(t) \geq e^{-\sqrt{NE}(t-T)}N_c(T)>0, \qquad t\geq T. \end{equation} As $t\to\infty$, \(N_c(t)\to N_c \) and \(\|F_t-F_{\rm be}\|_1\to 0. \)\end{theorem}
We next combine \cref{thm:conditional-main} with the existence and semi-strong relaxation theory of Cai and Lu \cite{CaiLu2026}. This gives a finite-time condensation result for arbitrary admissible initial data in the entire range $0\leq\eta<1$.
\begin{theorem}[Finite-time condensation for Cai--Lu solutions] \label{thm:CL-main} Assume \cref{ass:CL}. Let $F_0\in\mathcal B_1^+([0,\infty))$ satisfy \begin{equation}\label{eq:CL-initial-main} N(F_0)>0, \qquad E(F_0)>0, \qquad {\overline T}/{\overline T_c}<1, \end{equation} where $\overline T/\overline T_c$ is determined by $N=N(F_0)$ and $E=E(F_0)$ through \eqref{eq:temperature-ratio}. Then, for every fixed \(\lambda\in(1/{20},1/{19})\), there exists a conservative isotropic measure solution $(F_t)_{t\geq0}$ with initial datum $F_0$ such that \begin{equation}\label{eq:semistrong-rate-main} \|F_t-F_{\rm be}\|_1^\circ \leq C_\lambda(1+t)^{-\lambda/2}, \qquad t\geq0, \end{equation} for some constant $C_\lambda>0$. For this solution there exists a finite time $T\geq0$ such that \begin{equation}\label{eq:CL-atom-main} F_T(\{0\})>0,\qquad F_t(\{0\})>0, \quad t\geq T, \end{equation} and \begin{equation}\label{eq:CL-convergence-main} F_t(\{0\})\longrightarrow N_c, \qquad \|F_t-F_{\rm be}\|_1\longrightarrow0 \qquad\text{as }t\to\infty. \end{equation}\end{theorem}
\subsection{Strategy of the proof}\label{sec:intro-strategy}
Fix \(h\in(0,N_c)\).  For \(\varepsilon>0\), set
\[
 \varphi_\varepsilon(x)
 =\left(1-\frac{x}{\varepsilon}\right)_+^2,
 \qquad
 N_{0,2}(F,\varepsilon)
 =\int_{[0,\infty)}\varphi_\varepsilon(x)\,dF(x),
\]
as in \eqref{eq:cutoff-mass}.  For every fixed \(\varepsilon>0\),
we show that semi-strong relaxation to \(F_{\rm be}\), provided by \cite{CaiLu2026}, implies
\(N_{0,2}(F_t,\varepsilon)\geq h\) for all sufficiently large \(t\);
see \cref{lem:late-cutoff}.  We first choose a sufficiently small dyadic
scale \(\varepsilon_j\) and then a time \(t_0\) such that
\(N_{0,2}(F_{t_0},\varepsilon_j)\geq h\).  Writing
\(c_0=\sqrt{NE}\), the persistence estimate
\eqref{eq:cutoff-persistence} then gives
\[
 N_{0,2}(F_t,\varepsilon_j)
 \geq e^{-c_0(t-t_0)}h,
 \qquad t\geq t_0.
\]
To examine the collision dynamics below \(\varepsilon_j\), introduce
the dyadic energy scales and shells
\[
 \varepsilon_k=2^{-k-4},
 \qquad
 S_k=(\varepsilon_{k+1},\varepsilon_k],
 \qquad
 A_k(t)=F_t(S_k).
\]
Thus \(A_k(t)\) is the mass in the \(k\)-th shell.  We also choose
\[
 m_k=4\left(1+\log\frac1{\varepsilon_k}\right),
 \qquad
 s_k=\frac{\varepsilon_k}{4m_k},
\]
so that \(L_k=m_ks_k=\varepsilon_{k+2}\), and use the convex cutoff
\(
 \psi_{s,m}(x)=(e^{-x/s}-e^{-m})_+^2,
\)
whose properties are recorded in \cref{lem:profile}.  At scale \(k\), set
\[
 U_k(t)
 =\int_{[0,\infty)}\psi_{s_k,m_k}(x)\,dF_t(x)
\]
as in \eqref{eq:Uk-Ak}.  The function \(\psi_{s_k,m_k}\) is supported in
\([0,\varepsilon_{k+2}]\), strictly below \(S_k\), and satisfies
\(\psi_{s_k,m_k}\leq\varphi_{\varepsilon_{k+2}}\) by
\cref{lem:profile}.  Hence \(U_k(t)\)
is a weighted moment of the mass below the \(k\)-th shell and
\(
 U_k(t)\leq N_{0,2}(F_t,\varepsilon_{k+2}).
\)
For all sufficiently large \(k\), the one-scale estimate
\eqref{eq:relative-Uk-growth} gives
\[
 U_k'(t)+c_0U_k(t)
 \geq
 \gamma_kA_k(t)^2
 \left(U_k(t)+\frac{Q_k}{2}\right).
\]
Here \(\gamma_k>0\) and \(Q_k>0\) are the constants defined in
\eqref{eq:gamma-Q}, independent of time.  The term
\(\gamma_kA_k(t)^2Q_k/2\), obtained from the quadratic collision form
in \eqref{eq:quadratic-integrated}, is positive whenever \(A_k(t)>0\),
including when \(U_k(t)=0\).  The term proportional to \(U_k(t)\),
obtained from the cubic collision form after controlling the error as
in \eqref{eq:cubic-before-boundary} and \eqref{eq:boundary-error},
describes the Bose-enhanced transfer of particles from \(S_k\) toward
lower energy scales.
We next remove the linear term \(c_0U_k\).  Multiplying by the
integrating factor and using \(t\geq t_0\), we obtain
\[
 \frac{d}{dt}
 \left(e^{c_0(t-t_0)}U_k(t)\right)
 \geq
 \gamma_kA_k(t)^2
 \left(
 e^{c_0(t-t_0)}U_k(t)+\frac{Q_k}{2}
 \right).
\]
Consequently,
\[
 Z_k(t)
 =
 \log\left(
 1+\frac{2e^{c_0(t-t_0)}U_k(t)}{Q_k}
 \right)
\]
satisfies
\(
 Z_k'(t)\geq\gamma_kA_k(t)^2
\)
by \cref{lem:log-growth}; see also \eqref{eq:Z-def} and
\eqref{eq:Z-growth}.
For fixed \(h\), let
\[
 \Lambda_k(h)
 =\log\left(1+\frac{2h}{Q_k}\right),
\]
as in \eqref{eq:Lambda-def}.  We say that scale \(k\) is \(h\)-complete
at time \(t\) if \(Z_k(t)\geq\Lambda_k(h)\), which is equivalent to
\(e^{c_0(t-t_0)}U_k(t)\geq h\); see
\cref{def:complete-scale} and \eqref{eq:level-equivalence}.  Since
\(\psi_{s_k,m_k}\leq\varphi_{\varepsilon_{k+2}}\), an \(h\)-complete
scale satisfies
\[
 N_{0,2}(F_t,\varepsilon_{k+2})
 \geq U_k(t)
 \geq e^{-c_0(t-t_0)}h.
\]
Thus the estimate associated
with \(S_k\) gives the above lower bound for
\(N_{0,2}(F_t,\varepsilon_{k+2})\).  In this sense, the required
transfer from \(S_k\) toward lower energies has been completed.  Hence,
we use
\(
 \min\left\{1,{Z_k(t)}/{\Lambda_k(h)}\right\}
\)
as the completion level at scale \(k\): it is strictly less
than one before the target is reached and equals one after the scale
becomes \(h\)-complete.
To treat all sufficiently small scales \(k\geq j\) simultaneously, we
introduce
\[
 \omega_k(h)
 =\left(\frac{\Lambda_k(h)}{\gamma_k}\right)^{1/2},
 \qquad
 \Theta_j(h)
 =\sum_{k\geq j}\omega_k(h),
\]
as in \eqref{eq:weights}, and define
\[
 \mathcal P(t)
 =
 \sum_{k\geq j}\omega_k(h)
 \min\left\{
 1,\frac{Z_k(t)}{\Lambda_k(h)}
 \right\}
\]
as in \eqref{eq:Pcal}, to record the completion levels of all scales
\(k\geq j\).  Hence
\(
 0\leq\mathcal P(t)\leq\Theta_j(h).
\)
At every scale that is still \(h\)-incomplete,
\(Z_k(t)<\Lambda_k(h)\), and
\[
 \frac{d}{dt}
 \left(
 \omega_k(h)\frac{Z_k(t)}{\Lambda_k(h)}
 \right)
 \geq
 \frac{A_k(t)^2}{\omega_k(h)}.
\]
Summing over all \(h\)-incomplete scales gives
\[
 \mathcal P'(t)
 \geq
 \sum_{\substack{k\geq j\\
                  Z_k(t)<\Lambda_k(h)}}
 \frac{A_k(t)^2}{\omega_k(h)}
\]
for almost every \(t\) by \cref{lem:truncated-sum}; see
\eqref{eq:truncated-sum-derivative}.  Since
\(
 \omega_k(h)
 \asymp_{h,\eta}
 b_0^{-1/2}(1+k)\varepsilon_k^{(1-\eta)/2}
\)
and \(\eta<1\),
\[
 \Theta_j(h)<\infty,
 \qquad
 \Theta_j(h)\to 0
 \quad\text{as }j\to\infty.
\]
We now suppose
\(
 F_t(\{0\})
 <
 (1-\delta)e^{-c_0(t-t_0)}h
\)
for some \(\delta\in(0,1/2)\) throughout a short interval after \(t_0\);
cf. \eqref{eq:contrary-atom-bound}.  Then there can be only finitely many \(h\)-complete scales at any fixed time
in this interval.
If no scale \(k\geq j\) is \(h\)-complete, set \(\nu(t)=j\).  Otherwise,
let
\[
 \nu(t)
 =
 2+\max\{k\geq j:Z_k(t)\geq\Lambda_k(h)\},
\]
as in \eqref{eq:nu-definition}.  Persistence and \(h\)-completion then
give
\[
 N_{0,2}(F_t,\varepsilon_{\nu(t)})
 \geq e^{-c_0(t-t_0)}h;
\]
see \eqref{eq:nu-cutoff-bound}.  Hence
\[
 \begin{aligned}
 \sum_{k\geq\nu(t)}A_k(t)
 =F_t((0,\varepsilon_{\nu(t)}])\geq
 N_{0,2}(F_t,\varepsilon_{\nu(t)})-F_t(\{0\})\geq
 \delta e^{-c_0(t-t_0)}h.
 \end{aligned}
\]
 Thus a definite amount of
positive-energy mass lies in \(h\)-incomplete shells.  On such an
interval of length at most one, Cauchy--Schwarz and the lower bound for
\(\mathcal P'\) yield
\[
 \begin{aligned}
 \mathcal P'(t)\geq
 \sum_{k\geq\nu(t)}
 \frac{A_k(t)^2}{\omega_k(h)}\geq
 \frac{\left(\sum_{k\geq\nu(t)}A_k(t)\right)^2}
 {\sum_{k\geq\nu(t)}\omega_k(h)}\geq
 \frac{\delta^2e^{-2c_0}h^2}{\Theta_j(h)}.
 \end{aligned}
\]
 Set
\(
 \tau_j(h,\delta)
 =
 2e^{2c_0}\delta^{-2}h^{-2}\Theta_j(h)^2
\)
as in \eqref{eq:formation-interval}.  For sufficiently large \(j\), one
has \(\tau_j(h,\delta)\leq1\).  If
\(
 F_t(\{0\})
 <
 (1-\delta)e^{-c_0(t-t_0)}h
\)
held throughout \([t_0,t_0+\tau_j(h,\delta)]\), integration would give
\[
 \mathcal P(t_0+\tau_j(h,\delta))-\mathcal P(t_0)
 \geq2\Theta_j(h).
\]
This contradicts \(0\leq\mathcal P(t)\leq\Theta_j(h)\).  Hence a
positive atom must appear during this finite interval, and the
persistence estimate \eqref{eq:atom-persistence} keeps it positive
afterward; see \cref{thm:low-energy-to-condensate}, particularly
\eqref{eq:atom-lower-bound}.  Once the condensate has formed, a
sliding-time application of the same argument yields
\(F_t(\{0\})\to N_c\), and consequently
\(\norm{F_t-F_{\rm be}}_1\to0\); see
\eqref{eq:sliding-lower}--\eqref{eq:atom-liminf},
\eqref{eq:narrow-from-semistrong}, and \cref{lem:CL-strong}.
\subsection{Organization of the paper}\label{sec:intro-organization}
The rest of the paper is organized as follows. \Cref{sec:model} gives the isotropic weak formulation in the energy
variable and states the results from \cite{CaiLu2026} used in the
proof. \Cref{sec:kernel-estimates} proves lower bounds for the collision
kernel and for the quadratic and cubic collision forms.
\Cref{sec:one-scale} applies these bounds at one fixed energy scale.
\Cref{sec:multiscale} introduces \(Z_k\) at each dyadic scale and
combines the resulting estimates in the weighted sum \(\mathcal P\).
Finally, \cref{sec:atom-formation} proves finite-time formation of a
zero-energy atom, convergence of its mass, and strong convergence to
equilibrium.
\section{Isotropic weak formulation and relaxation estimates}\label{sec:model}
\subsection{The reduced collision operator}
We first expand the Bose factors as
\begin{align*}
 f'f_*'(1+f)(1+f_*)-ff_*(1+f')(1+f_*')
 =(f'f_*'-ff_*)
 +\bigl[f'f_*'(f+f_*)-ff_*(f'+f_*')\bigr],
\end{align*}
and then the quartic terms cancel identically. The first term \(f'f_*'-ff_*\) is the
quadratic gain--loss contribution in usual Boltzmann equations, whereas the second term \(f'f_*'(f+f_*)-ff_*(f'+f_*')\) is cubic,
representing the Bose enhancement of transitions into occupied states.
For isotropic distributions, the collision operator can be written
using the particle energies alone. After integrating over the angular
variables, we can use a reduced kernel \(W\) to describe the collision geometry and the scattering factor \(\Phi\). The
quadratic and cubic contributions then are encoded respectively in
weak collision forms \(\mathcal J\) and \(\mathcal K\) constructed from
this kernel; see \cite[(1.9)--(1.14)]{CaiLu2026}. We now make this
representation precise, beginning with the collision energy variables
and the definition of \(W\).
For the collision variables, take
\[
 x=\frac{|v|^2}{2},\qquad
 y=\frac{|v'|^2}{2},\qquad
 z=\frac{|v_*'|^2}{2},\qquad
 x_*=\frac{|v_*|^2}{2}=(y+z-x)_+.
\]
For fixed $x,y,z$, set $|v-v'|=\sqrt2\,\sigma$.
Since $|v-v'|=|v_*-v_*'|$, the triangle inequalities give
\begin{equation}\label{eq:sigma-pm}
 \sigma_-=
 \max\bigl\{|\sqrt x-\sqrt y|,|\sqrt{x_*}-\sqrt z|\bigr\},
 \qquad
 \sigma_+=
 \min\bigl\{\sqrt x+\sqrt y,\sqrt{x_*}+\sqrt z\bigr\}.
\end{equation}
We use the reduced-kernel representation in
\cite[(1.13)--(1.14)]{CaiLu2026}. For
$\sigma\in[\sigma_-,\sigma_+]$ and $\theta\in[0,2\pi]$, set
\begin{equation}\label{eq:Ystar}
 Y_*=
 \begin{cases}
 \displaystyle\left|
 \sqrt{\left(z-\frac{(x-y+\sigma^2)^2}{4\sigma^2}\right)_+}
 +e^{\mathrm i\theta}
 \sqrt{\left(x-\frac{(x-y+\sigma^2)^2}{4\sigma^2}\right)_+}
 \right|,
 &\sigma>0,\\[1.5ex]
 0,&\sigma=0.
 \end{cases}
\end{equation}
Thus $|v-v_*'|=\sqrt2\,Y_*$. For $x_*xyz>0$, define
\begin{equation}\label{eq:W}
 W(x,y,z)=\frac{1}{4\pi\sqrt{xyz}}
 \int_{\sigma_-}^{\sigma_+}\int_0^{2\pi}
 \Phi(\sqrt2\,\sigma,\sqrt2\,Y_*)\,d\theta\,d\sigma.
\end{equation}
For $x_*xyz=0$, set
\begin{equation}\label{eq:W-boundary}
 W(x,y,z)=
 \begin{cases}
 \displaystyle\frac{\Phi(\sqrt{2y},\sqrt{2z})}{\sqrt{yz}},
 &x=0,\quad y>0,\quad z>0,\\[1.2ex]
 \displaystyle\frac{\Phi(\sqrt{2x},\sqrt{2(z-x)})}{\sqrt{xz}},
 &y=0,\quad 0<x\leq z,\\[1.2ex]
 \displaystyle\frac{\Phi(\sqrt{2(y-x)},\sqrt{2x})}{\sqrt{xy}},
 &z=0,\quad 0<x\leq y,\\[1.2ex]
 0,&\text{otherwise}.
 \end{cases}
\end{equation}
Let $W_H$ denote the hard-sphere kernel obtained from $W$ by setting
$\Phi\equiv1$.  For $x_*xyz>0$,
\begin{equation}\label{eq:WH}
 W_H(x,y,z)=
 \frac{\min\{\sqrt x,\sqrt{x_*},\sqrt y,\sqrt z\}}
 {\sqrt{xyz}}.
\end{equation}
For $x_*xyz=0$, its values are given by
\eqref{eq:W-boundary} with $\Phi\equiv1$
\cite[(1.13), (1.15)--(1.16)]{CaiLu2026}.
\begin{lemma}[Basic properties of \(W\)]
\label{lem:kernel-basic}
The reduced kernel $W$ is a nonnegative Borel function on
$[0,\infty)^3$ and satisfies
\begin{equation}\label{eq:W-upper-global}
 0\leq W(x,y,z)\leq W_H(x,y,z),
 \qquad (x,y,z)\in[0,\infty)^3.
\end{equation}
Moreover,
\begin{equation}\label{eq:W-last-two-symmetry}
 W(x,y,z)=W(x,z,y),\qquad x,y,z\geq0.
\end{equation}
\end{lemma}
\begin{proof}
The functions \(x_*=(y+z-x)_+\) and \(\sigma_\pm\) are continuous in
\((x,y,z)\). By \eqref{eq:Ystar}, \(Y_*\) is continuous when
\(\sigma>0\), and its value at \(\sigma=0\) makes it Borel
measurable. Hence the function
\[
 (x,y,z,\sigma,\theta)\longmapsto
 \begin{cases}
 \displaystyle
 \frac{\Phi(\sqrt2\,\sigma,\sqrt2\,Y_*)}{4\pi\sqrt{xyz}},
 &x_*xyz>0,\quad \sigma_-\leq\sigma\leq\sigma_+,\\[1.2ex]
 0,&\text{otherwise},
 \end{cases}
\]
is jointly Borel on
\([0,\infty)^4\times[0,2\pi]\). Its integral in
\((\sigma,\theta)\) is therefore Borel in \((x,y,z)\). This proves the
measurability of \(W\) on \(\{x_*xyz>0\}\). Since each nonzero expression in
\eqref{eq:W-boundary} is continuous on its stated domain,
\eqref{eq:W-boundary} shows the measurability of \(W\) on
\(\{x_*xyz=0\}\). Thus \(W\) is Borel on \([0,\infty)^3\), and
\(W\geq0\) follows from \(\Phi\geq0\). If \(x_*xyz>0\), then \(0\leq\Phi\leq1\) and \eqref{eq:W} give
\[
 0\leq W(x,y,z)
 \leq
 \frac{1}{4\pi\sqrt{xyz}}
 \int_{\sigma_-}^{\sigma_+}\int_0^{2\pi}
 1\,d\theta\,d\sigma
 =W_H(x,y,z).
\]
On \(\{x_*xyz=0\}\), the same inequality follows directly from
\eqref{eq:W-boundary} and the corresponding boundary values of
\(W_H\). This proves \eqref{eq:W-upper-global}.
For \(x_*xyz>0\), \cite[Lemma~6.3, (6.8)]{CaiLu2019} and the symmetry
of \(\Phi\) give directly
\[
 W(x,y,z)=W(x,z,y).
\]
On \(\{x_*xyz=0\}\), the same identity follows from
\eqref{eq:W-boundary}: the first nonzero expression is symmetric in
\(y,z\), the second and third expressions are exchanged by
\(y\leftrightarrow z\), and all remaining values are zero. This proves
\eqref{eq:W-last-two-symmetry}.
\end{proof}
We now
fix the admissible test spaces. We write
\[
 C_b^1([0,\infty))
 =\{\varphi\in C^1([0,\infty)):\varphi,\varphi'
 \text{ are bounded}\},
\]
and
\[
 C_b^{1,1}([0,\infty))
 =\{\varphi\in C_b^1([0,\infty)):\varphi'\text{ is globally
 Lipschitz on }[0,\infty)\}.
\]
We also write $C_b^2([0,\infty))$ for the space of functions
$\varphi\in C^2([0,\infty))$ such that $\varphi$, $\varphi'$, and
$\varphi''$ are bounded.
The mean-value theorem gives
$C_b^2([0,\infty))\subset C_b^{1,1}([0,\infty))$.
For $\varphi\in C_b^{1,1}([0,\infty))$ define
\begin{align}
 \Delta\varphi(x,y,z)
 &=\varphi(x)+\varphi(x_*)-\varphi(y)-\varphi(z),\label{eq:Delta}\\
 \mathcal K[\varphi](x,y,z)
 &=W(x,y,z)\Delta\varphi(x,y,z),\label{eq:K}\\
 \mathcal J[\varphi](y,z)
 &=\frac12\int_0^{y+z}\mathcal K[\varphi](x,y,z)\sqrt{x}\,dx.
 \label{eq:J}
\end{align}
\begin{lemma}[Well-definedness of \(\mathcal K,\mathcal J\)]
\label{lem:collision-forms-well-defined}
For every $\varphi\in C_b^{1,1}([0,\infty))$,
$\mathcal K[\varphi]$ is bounded and Borel measurable. For every $y,z\geq0$,
the integral defining $\mathcal J[\varphi](y,z)$ is integrable. Moreover, $\mathcal J[\varphi]$ is Borel measurable on
$[0,\infty)^2$ and satisfies
\begin{equation}\label{eq:J-well-defined-bound}
 |\mathcal J[\varphi](y,z)|
 \leq4\norm{\varphi}_\infty(\sqrt y+\sqrt z),
 \qquad y,z\geq0.
\end{equation}
\end{lemma}
\begin{proof}
Since $x_*=(y+z-x)_+$ is continuous, $\Delta\varphi$ is continuous in
$(x,y,z)$, and $W$ is Borel measurable by the previous lemma, $\mathcal K[\varphi]=W\Delta\varphi$ is Borel measurable. Its boundedness is recorded in
\cite[Sec.~I.A]{CaiLu2026}. 
We first prove integrability of \(\mathcal J[\varphi]\). Suppose that $0<y\leq z$. For
$0<x<y+z$, one has $x_*=y+z-x>0$, so
\eqref{eq:W-upper-global} and \eqref{eq:WH} give
\[
 W(x,y,z)\sqrt x
 \leq W_H(x,y,z)\sqrt x
 =\frac{\min\{\sqrt x,\sqrt{x_*},\sqrt y,\sqrt z\}}
        {\sqrt{yz}}
 \leq\frac1{\sqrt z}.
\]
Since $|\Delta\varphi|\leq4\norm{\varphi}_\infty$ and
$y+z\leq2z$, it follows that
\begin{align*}
 \frac12\int_0^{y+z}
 |\mathcal K[\varphi](x,y,z)|\sqrt x\,dx\leq
 \frac{2\norm{\varphi}_\infty}{\sqrt z}(y+z)\leq4\norm{\varphi}_\infty\sqrt z.
\end{align*}
The same argument with $y$ and $z$ interchanged covers the case
$0<z<y$. If $y=0<z$, then \eqref{eq:W-upper-global} and the boundary value of
$W_H$ give
\[
 W(x,0,z)\sqrt x\leq\frac1{\sqrt z},
 \qquad 0<x\leq z.
\]
Therefore
\[
 \frac12\int_0^z
 |\mathcal K[\varphi](x,0,z)|\sqrt x\,dx
 \leq2\norm{\varphi}_\infty\sqrt z.
\]
The case $z=0<y$ is analogous, and the integral is zero when
$y=z=0$. These estimates prove integrability of \(\mathcal J[\varphi]\) and
\eqref{eq:J-well-defined-bound}. Finally,
\[
 (x,y,z)\mapsto
 \frac12\mathbf 1_{\{0\leq x\leq y+z\}}
 \mathcal K[\varphi](x,y,z)\sqrt x
\]
is jointly Borel measurable and absolutely
integrable in $x$ for every $y,z\geq0$. The measurability theorem for
parameter integrals therefore shows that
\[
 (y,z)\mapsto
 \frac12\int_0^{y+z}
 \mathcal K[\varphi](x,y,z)\sqrt x\,dx
 =\mathcal J[\varphi](y,z)
\]
is Borel measurable.
\end{proof}
Throughout this paper, when the variables are clear, we write
$dF^{\otimes2}=dF(y)\,dF(z)$.  Similarly,
$dF^{\otimes3}=dF(x)\,dF(y)\,dF(z)$.
\subsection{Measure solutions and the weak formulation}
With $W$, $\Delta\varphi$, $\mathcal K$, and $\mathcal J$ already introduced,
we can now state the definition of measure solutions to the equation \eqref{eq:BN}.  The definition uses $C_b^2$ tests in \cite{CaiLu2026}, and we then extend the test function class to
$C_b^{1,1}$ needed in the subsequent analysis.
Following \cite[Definition~1.1]{CaiLu2026}, we call a family
$(F_t)_{t\geq0}\subset\mathcal B_1^+([0,\infty))$ a \textit{conservative isotropic
measure solution} if
\[
 N(F_t)=N,\qquad E(F_t)=E,\qquad t\geq0,
\]
and if, for every $\varphi\in C_b^2([0,\infty))$, the map
$\displaystyle t\mapsto\int_{[0,\infty)}\varphi(x)\,dF_t(x)$ is $C^1$ and, for every
$t\geq0$, satisfies
\begin{equation}\label{eq:weak}
 \frac{d}{dt}\int\varphi\,dF_t
 =\iint\mathcal J[\varphi]\,dF_t^{\otimes2}
 +\iiint\mathcal K[\varphi]\,dF_t^{\otimes3}.
\end{equation}
Here and below, every unmarked multiple integral is taken over
$[0,\infty)$ in each variable.
\begin{lemma}[Narrow continuity of \((F_t)_{t\geq 0}\)]
\label{lem:narrow-continuity}
Every conservative isotropic measure solution $(F_t)_{t\geq0}$ is
narrowly continuous.  Equivalently, for every
$\varphi\in C_b([0,\infty))$, the map
$\displaystyle t\mapsto\int\varphi\,dF_t$ is continuous.
\end{lemma}
\begin{proof}
Fix $s\geq0$ and $\varphi\in C_b([0,\infty))$. For every $R>0$,
mass and energy conservation give
\[
 F_t([0,\infty))=N,
 \qquad
 F_t((R,\infty))
 \leq\frac1R\int_{(R,\infty)}x\,dF_t(x)
 \leq\frac ER,
 \qquad t\geq0.
\]
Given $\varepsilon>0$, one can then
choose $\psi\in C_b^2([0,\infty))$ such that
\[
 \sup_{0\leq x\leq R}|\varphi(x)-\psi(x)|<\varepsilon,
 \qquad
 \|\psi\|_\infty\leq\|\varphi\|_\infty.
\]
Splitting the error over $[0,R]$ and $(R,\infty)$ gives
\begin{align*}
 \left|\int\varphi\,d(F_t-F_s)\right|
 &\leq
 \left|\int\psi\,d(F_t-F_s)\right|
 +\int|\varphi-\psi|\,dF_t
 +\int|\varphi-\psi|\,dF_s\\
 &\leq
 \left|\int\psi\,d(F_t-F_s)\right|
 +2N\varepsilon
 +4\|\varphi\|_\infty\frac ER.
\end{align*}
Since $\psi\in C_b^2([0,\infty))$, the definition of a measure
solution implies
\[
 \int\psi\,dF_t\longrightarrow\int\psi\,dF_s
 \qquad\text{as }t\to s.
\]
Consequently,
\[
 \limsup_{t\to s}
 \left|\int\varphi\,d(F_t-F_s)\right|
 \leq
 2N\varepsilon+4\|\varphi\|_\infty\frac ER.
\]
Letting $\varepsilon\downarrow0$ and then $R\to\infty$ proves the
continuity at $s$. Since $s\geq0$ was arbitrary, the proof is complete.
\end{proof}
\begin{lemma}[Extension to $C_b^{1,1}$ test functions]
\label{lem:test-extension}
Under \cref{ass:CL}, let $(F_t)_{t\geq0}$ be a conservative isotropic
measure solution with conserved mass $N$ and energy $E$.  For every
$\varphi\in C_b^{1,1}([0,\infty))$, the map
\[
 t\longmapsto\int_{[0,\infty)}\varphi(x)\,dF_t(x)
\]
belongs to $C^1([0,\infty))$, and \eqref{eq:weak} holds for every
$t\geq0$.  Consequently, for all $0\leq s\leq t<\infty$,
\begin{equation}\label{eq:weak-integrated-C11}
 \int\varphi\,d(F_t-F_s)
 =
 \int_s^t\left(
 \iint\mathcal J[\varphi]\,dF_\tau^{\otimes2}
 +
 \iiint\mathcal K[\varphi]\,dF_\tau^{\otimes3}
 \right)d\tau.
\end{equation}
\end{lemma}
\begin{proof}
Cai and Lu record that the test-function class in their
Definition~1.1 can be enlarged from $C_b^2([0,\infty))$ to
$C_b^{1,1}([0,\infty))$; see
\cite[Sec.~I.A, paragraph following (1.16)]{CaiLu2026}.
The assumptions on $\Phi$ used there are implied by \cref{ass:CL}.
\end{proof}
For later use, conservation of mass and energy and Cauchy--Schwarz give
\begin{equation}\label{eq:extension-half-moment}
 \int_{[0,\infty)}\sqrt{x}\,dF_t(x)
 \leq
 \left(\int dF_t\right)^{1/2}
 \left(\int x\,dF_t(x)\right)^{1/2}
 =\sqrt{NE},
 \qquad t\geq0.
\end{equation}
\subsection{Localized mass near zero} For $\varepsilon>0$, set
$\varphi_\varepsilon(x)=(1-x/\varepsilon)_+^2$. Following the
notation in
\cite[Sec.~III, notation preceding Lemma~3.1]{CaiLu2026}, define
\begin{equation}\label{eq:cutoff-mass}
 N_{0,2}(F,\varepsilon)
 =\int_{[0,\infty)}\varphi_\varepsilon(x)\,dF(x)
 =\int_{[0,\varepsilon]}
 \left(1-\frac{x}{\varepsilon}\right)^2\,dF(x).
\end{equation}
Since
$\varphi_\varepsilon'(x)=-(2/\varepsilon)(1-x/\varepsilon)_+$, $\varphi_\varepsilon$ is nonnegative, nonincreasing, and
convex, with
$\operatorname{Lip}(\varphi_\varepsilon')=2/\varepsilon^2$.
In particular,
$\varphi_\varepsilon\in C_b^{1,1}([0,\infty))$ and may be used as a
test function in \cref{lem:test-extension}.
If $0<\varepsilon_1<\varepsilon_2$, then
$\varphi_{\varepsilon_1}\leq\varphi_{\varepsilon_2}$.  Moreover,
$\varphi_\varepsilon(x)\downarrow\mathbf 1_{\{0\}}(x)$ for every
$x\geq0$ as $\varepsilon\downarrow0$.  Hence, for every
$F\in\mathcal B_1^+([0,\infty))$, dominated convergence gives
\begin{equation}\label{eq:cutoff-limit}
 N_{0,2}(F,\varepsilon)\downarrow F(\{0\})
 \qquad\text{as }\varepsilon\downarrow0.
\end{equation}
\begin{lemma}[Comparison of low-energy cutoff masses]\label{lem:cutoff-comparison}
If $F,G\in\mathcal B_1^+([0,\infty))$ satisfy \(N(F)=N(G)\), then
\begin{equation}\label{eq:cutoff-comparison}
 \abs{N_{0,2}(F,\varepsilon)-N_{0,2}(G,\varepsilon)}
 \leq\frac2\varepsilon\norm{F-G}_1^\circ.
\end{equation}
\end{lemma}
\begin{proof}
Since \(N(F)=N(G)\),
$$\displaystyle \int_{[0,\infty)}1\,d(F-G)=0.$$  Therefore
\[
 \begin{aligned}
 \abs{N_{0,2}(F,\varepsilon)-N_{0,2}(G,\varepsilon)}
 =\left|\int_{[0,\infty)}(\varphi_\varepsilon-1)\,d(F-G)\right|\leq\int_{[0,\infty)}|1-\varphi_\varepsilon|\,d|F-G|.
 \end{aligned}
\]
For $0\leq x\leq\varepsilon$, one has
$1-\varphi_\varepsilon(x)=2x/\varepsilon-x^2/\varepsilon^2
\leq2x/\varepsilon$, whereas for $x>\varepsilon$,
$1-\varphi_\varepsilon(x)=1\leq x/\varepsilon$.  Hence
\[
 \abs{N_{0,2}(F,\varepsilon)-N_{0,2}(G,\varepsilon)}
 \leq\frac2\varepsilon
 \int_{[0,\infty)}x\,d|F-G|(x)
 =\frac2\varepsilon\norm{F-G}_1^\circ.
\]
\end{proof}
\subsection{Semi-strong relaxation and strong convergence criterion}
We record the two consequences of Cai and Lu used below.
\begin{proposition}[Cai--Lu semi-strong relaxation]
\label{prop:CL-relaxation}
Under \cref{ass:CL}, let
$F_0\in\mathcal B_1^+([0,\infty))$ satisfy
$N=N(F_0)>0$ and $E=E(F_0)>0$.  For every
$\lambda\in(1/20,1/19)$, there exist a conservative isotropic measure
solution $(F_t)_{t\geq0}$ with initial datum $F_0$ and a constant
$C_\lambda>0$ such that
\begin{equation}\label{eq:semistrong-rate}
 \norm{F_t-F_{\mathrm{be}}}_1^\circ
 \leq C_\lambda(1+t)^{-\lambda/2},
 \qquad t\geq0,
\end{equation}
where $F_{\mathrm{be}}$ is the equilibrium with mass $N$ and energy $E$,
and $C_\lambda$ depends only on $N,E,b_0,\eta$, and $\lambda$.
\end{proposition}
\begin{proof}
The hypotheses of \cite[Theorem~2.2]{CaiLu2026} follow from
\cref{ass:CL}.  Since $0\leq\eta<1$,
$\max\{2,(4+2\eta)/3\}=2$, so the admissible exponent range in that
theorem is $1/20<\lambda<1/19$.  Combining
\cite[Theorem~2.2]{CaiLu2026} with
\cite[Lemma~2.1, (2.1)]{CaiLu2026} gives
\eqref{eq:semistrong-rate}.
\end{proof}
\begin{lemma}[Cai--Lu strong-convergence criterion]
\label{lem:CL-strong}
Suppose $N_c>0$, and let
$F\in\mathcal B_1^+([0,\infty))$ have the same mass and energy as
$F_{\mathrm{be}}$.  Then
\begin{equation}\label{eq:CL-strong-criterion}
 \abs{F(\{0\})-N_c}
 \leq\norm{F-F_{\mathrm{be}}}_1
 \leq
 2\abs{F(\{0\})-N_c}
 +C\bigl(\norm{F-F_{\mathrm{be}}}_1^\circ\bigr)^{1/3},
\end{equation}
where $C$ depends only on the common mass and energy.
\end{lemma}
\begin{proof}
By \eqref{eq:Nc}, $N_c>0$ is equivalent to
$\overline T/\overline T_c<1$.  The claim follows from
\cite[Lemma~2.1, (2.2)]{CaiLu2026}.
\end{proof}
\section{Estimates for the collision operator}
\label{sec:kernel-estimates}
Throughout this section we assume \cref{ass:CL}.
\subsection{Positivity of the integral of
\texorpdfstring{\(\mathcal K[\varphi]\)}{K[phi]}
for convex test functions}
For $0\leq x<y\leq z$, define
\begin{equation}\label{eq:Delta-sym}
 \Delta_{\mathrm{sym}}\varphi(x,y,z)
 =\varphi(z+y-x)+\varphi(z+x-y)-2\varphi(z),
\end{equation}
and let $\chi_{y,z}=2$ for $y<z$ and $\chi_{y,z}=1$ for $y=z$.
For every finite nonnegative Borel measure $F$ and every Borel interval
$I\subset[0,\infty)$,
\begin{equation}\label{eq:ordered-square}
 \int_{\{(y,z)\in I^2:\,y\leq z\}}
 \chi_{y,z}\,dF(y)\,dF(z)=F(I)^2.
\end{equation}
Indeed, symmetry of $F\otimes F$ gives
\[
 F(I)^2
 =2(F\otimes F)\bigl(\{(y,z)\in I^2:y<z\}\bigr)
 +(F\otimes F)\bigl(\{(y,z)\in I^2:y=z\}\bigr).
\]
\begin{lemma}[Cai--Lu convex positivity]
\label{lem:convex-positivity}
Under \cref{ass:CL}, let
$F\in\mathcal B_1^+([0,\infty))$.  If
$\varphi\in C_b^{1,1}([0,\infty))$ is convex, then
\begin{equation}\label{eq:convex-positivity}
 \int\mathcal K[\varphi]\,dF^{\otimes3}
 \geq
 \int_{\{0\leq x<y\leq z\}}
 \chi_{y,z}W\Delta_{\mathrm{sym}}\varphi\,dF^{\otimes3}
 \geq0.
\end{equation}
\end{lemma}
\begin{proof}
The conclusion is
\cite[Lemma~3.1(II), (3.4)]{CaiLu2026}.
\end{proof}
\subsection{Lower bound for
\texorpdfstring{\(\mathcal J[\varphi]\)}{J[phi]}
and persistence estimates}
For convex test functions, the cubic contribution is nonnegative by
\cref{lem:convex-positivity}, but the quadratic form $\mathcal J$ may be
negative.  The following estimate bounds its negative part.
\begin{lemma}[Cai--Lu quadratic lower bound]\label{lem:quadratic-lower}
Assume \cref{ass:CL}.  Let
$\varphi\in C_b^{1,1}([0,\infty))$ be nonnegative and convex.  Then
\begin{equation}\label{eq:J-pointwise-lower}
 \mathcal J[\varphi](y,z)
 \geq-\frac12(\varphi(y)\sqrt z+\varphi(z)\sqrt y),
 \qquad y,z\geq0.
\end{equation}
Consequently, if $F\in\mathcal B_1^+([0,\infty))$ has mass $N$ and energy
$E$, then
\begin{equation}\label{eq:J-integrated-lower}
 \iint\mathcal J[\varphi]\,dF^{\otimes2}
 \geq-\sqrt{NE}\int\varphi\,dF.
\end{equation}
\end{lemma}
\begin{proof}
The pointwise estimate \eqref{eq:J-pointwise-lower} is
\cite[Lemma~5.2, (5.13)]{CaiLu2019}.  Its integrability follows from
\cref{lem:collision-forms-well-defined}.  Integrating the pointwise estimate
and using \eqref{eq:extension-half-moment}, we obtain
\begin{align*}
 \iint\mathcal J[\varphi]\,dF^{\otimes2}
 &\geq-\frac12\iint
 \bigl(\varphi(y)\sqrt z+\varphi(z)\sqrt y\bigr)
 \,dF(y)\,dF(z)\\
 &=-\left(\int\varphi\,dF\right)
 \left(\int\sqrt x\,dF(x)\right)
 \geq-\sqrt{NE}\int\varphi\,dF.
\end{align*}
\end{proof}
\begin{lemma}[Persistence of cutoff moments and the condensate]
\label{lem:persistence}
Assume \cref{ass:CL}.  Let $F_t$ be a conservative isotropic measure
solution with mass $N>0$, energy $E>0$, and $c_0=\sqrt{NE}$.  For every
bounded, nonnegative, convex
$\varphi\in C_b^{1,1}([0,\infty))$,
\begin{equation}\label{eq:persistence}
 \int_{[0,\infty)}\varphi(x)\,dF_t(x)
 \geq e^{-c_0(t-s)}
 \int_{[0,\infty)}\varphi(x)\,dF_s(x),
 \qquad 0\leq s\leq t.
\end{equation}
In particular,
\begin{equation}\label{eq:cutoff-persistence}
 N_{0,2}(F_t,\varepsilon)
 \geq e^{-c_0(t-s)}N_{0,2}(F_s,\varepsilon),
 \qquad \varepsilon>0,
\end{equation}
and
\begin{equation}\label{eq:atom-persistence}
 F_t(\{0\})
 \geq e^{-c_0(t-s)}F_s(\{0\}).
\end{equation}
\end{lemma}
\begin{proof}
The integrating-factor estimate in \cite[Lemma~3.2]{CaiLu2026} gives
\eqref{eq:persistence}.  Taking $p=2$ in its cutoff-moment conclusion gives
\eqref{eq:cutoff-persistence}, and its conclusion for $F_t(\{0\})$ gives
\eqref{eq:atom-persistence}.
\end{proof}
\subsection{Uniform lower bound of
\texorpdfstring{\(W\)}{W}
for low-energy collisions}
\begin{lemma}[Low-energy reduced-kernel bound]
\label{lem:W-lower}
Under \cref{ass:CL}, for
$0\leq x<y\leq z\leq\varepsilon\leq1/16$,
\begin{equation}\label{eq:W-lower}
 W(x,y,z)\geq\frac{b_0}{2}\varepsilon^{\eta-1}.
\end{equation}
\end{lemma}
\begin{proof}
By \cite[Lemma~3.1(II), (3.12)]{CaiLu2026},
\[
 W(x,y,z)\geq\frac{2b_0}{3}
 \frac{(1\wedge z)^\eta}{\sqrt{yz}}
 =\frac{2b_0}{3}\frac{z^\eta}{\sqrt{yz}}.
\]
Since $\sqrt{yz}\leq z\leq\varepsilon$ and $\eta<1$,
\[
 \frac{z^\eta}{\sqrt{yz}}
 \geq z^{\eta-1}
 \geq\varepsilon^{\eta-1}.
\]
Because $2/3\geq1/2$, this proves \eqref{eq:W-lower}.
\end{proof}
\section{One-scale estimate for low-energy mass}\label{sec:one-scale}
To prove condensation, we need to understand the behavior of the solution
at arbitrarily small energies.  In \Cref{sec:multiscale}, we
introduce the dyadic low-energy scales
$\varepsilon_k=2^{-k-4}$ with the corresponding shells
$S_k=(\varepsilon_{k+1},\varepsilon_k]$, and study the transfer of mass
from these shells toward lower energies.  Since this requires us to treat
infinitely many low-energy scales simultaneously, we first establish the
basic estimate at one fixed scale.
Fix a sufficiently small energy level $\varepsilon\in (0,1/16]$, choose $s>0$ and $m\geq4$, and set $L=ms\leq\varepsilon$. Let $(F_t)_{t\geq0}$ be a conservative isotropic measure solution with mass $N$ and energy $E$.  We use the convex cutoff $\psi_{s,m}$ supported in $[0,L]$ and consider its moment
\[U(t)=\int\psi_{s,m}\,dF_t,\] which records a weighted amount of mass below
$L$.  We also write $A(t)=F_t([L,\varepsilon])$ for the mass in the
higher-energy interval.  By applying the weak formulation to
$\psi_{s,m}$, we obtain a differential inequality that describes how the
mass in $[L,\varepsilon]$ produces and enhances low-energy mass below
$L$.
More precisely, we show that the estimates of the collision operator
in \Cref{sec:kernel-estimates} give
\[
\begin{aligned}
 U'(t)
 &\geq
 \frac{(e^{-1/2}-e^{-4})^2}{12\sqrt2}\,
 b_0\varepsilon^{\eta-1}A(t)^2s^{3/2}
 -\sqrt{NE}\,U(t)\\
 &\quad+
 \frac{b_0}{2(e+1)^2}\,
 \varepsilon^{\eta-1}A(t)^2
 \frac{s}{\varepsilon+s}
 \bigl(U(t)-(e^2-1)^2Ne^{-2m}\bigr)_+ .
\end{aligned}
\]
We briefly introduce the origins of the different terms on the right-hand side of the inequality.
\begin{enumerate}[label=\textup{(\roman*)},leftmargin=*,itemsep=2pt]
\item
The term
\(\dfrac{(e^{-1/2}-e^{-4})^2}{12\sqrt2}
b_0\varepsilon^{\eta-1}A(t)^2s^{3/2}\)
comes from the gain part of the quadratic collision form.  It follows
from \cref{lem:W-lower} and \cref{lem:profile}(i); see
\eqref{eq:quadratic-pointwise}--\eqref{eq:quadratic-integrated}.
\item
The term \(-\sqrt{NE}\,U(t)\) controls the loss part of the quadratic
collision form.  It follows from
\cref{lem:quadratic-lower}; see \eqref{eq:quadratic-remainder} and
\eqref{eq:quadratic-integrated}.
\item
The term
\(\dfrac{b_0}{2(e+1)^2}\varepsilon^{\eta-1}A(t)^2
\dfrac{s}{\varepsilon+s}
\bigl(U(t)-(e^2-1)^2Ne^{-2m}\bigr)_+\)
comes from the cubic collision form and describes the Bose-enhanced
transfer of mass toward lower energies.  It follows from \cref{lem:convex-positivity},
\cref{lem:profile}(ii), and
\eqref{eq:cell-square-lower}; see \eqref{eq:cubic-before-boundary}.
\item
The correction \((e^2-1)^2Ne^{-2m}\) is the mass loss of the cutoff test function in the
boundary region \((L-2s,L]\).  It is controlled by
\cref{lem:profile}(iii); see \eqref{eq:boundary-positive-part}.
\end{enumerate}
Thus the inequality consists of a quadratic source, a linear quadratic
loss, and a cubic feedback term proportional to the cutoff mass.  More
specifically, the quadratic gain term initiates the transfer of mass from
the higher-energy shell $[L,\varepsilon]$ into the lower-energy region
below $L$, while the cubic term amplifies this transfer through Bose
enhancement once low-energy mass is present.
\subsection{Cutoff test function
\texorpdfstring{\(\psi_{s,m}\)}{psi(s,m)}}
For \(s>0\) and \(m\geq4\), define
\[
 \psi_{s,m}(x)=(e^{-x/s}-e^{-m})_+^2,
 \qquad L=ms.
\]
\begin{lemma}[Properties of \(\psi_{s,m}\)]
\label{lem:profile}
The function \(\psi_{s,m}\) is bounded, nonnegative, nonincreasing,
convex, and belongs to \(C_b^{1,1}([0,\infty))\), with \(\operatorname{supp}\psi_{s,m}\subset[0,L]\).  Moreover,
\begin{enumerate}[label=\textup{(\roman*)}]
\item
for \(0\leq x\leq s/2\), \(
 \psi_{s,m}(x)\geq(e^{-1/2}-e^{-4})^2\);
\item if \(0\leq d\leq s\) and \(0\leq x\leq L-2s\), then
\[
 \psi_{s,m}(x+d)\geq\frac{1}{(e+1)^2}\psi_{s,m}(x);
\]
\item for every finite nonnegative Borel measure \(F\) of mass \(N\),
\[
 \int_{(L-2s,\infty)}\psi_{s,m}(x)\,dF(x)
 \leq (e^2-1)^2Ne^{-2m};
\]
\item if \(L\leq\varepsilon'\), for \(x\ge 0\),
\(
 \psi_{s,m}(x)\leq\varphi_{\varepsilon'}(x).
\)
\end{enumerate}
\end{lemma}
\begin{proof}
Write $\psi=\psi_{s,m}$.  On $0\leq x<L$,
\[
 \psi'(x)
 =-\frac{2}{s}e^{-x/s}(e^{-x/s}-e^{-m}),
 \qquad
 \psi''(x)
 =\frac{2}{s^2}e^{-x/s}(2e^{-x/s}-e^{-m}).
\]
Since $e^{-x/s}>e^{-m}$ on $[0,L)$, we have $\psi'\leq0$ and
$\psi''\geq0$ there.  Moreover, $\psi''\leq4s^{-2}$,
while $\psi(L-)=\psi'(L-)=0$.  Since $\psi$ vanishes on $[L,\infty)$,
its derivative extends continuously by zero across $L$ and is globally
Lipschitz.  Hence $\psi\in C_b^{1,1}([0,\infty))$ is nonnegative,
nonincreasing, and convex, with
$\operatorname{supp}\psi\subset[0,L]$.
For \textup{(i)}, if $0\leq x\leq s/2$, then monotonicity and $m\geq4$
give
\[
 \psi(x)\geq\psi(s/2)
 =(e^{-1/2}-e^{-m})^2
 \geq(e^{-1/2}-e^{-4})^2.
\]
For \textup{(ii)}, let $0\leq d\leq s$ and $0\leq x\leq L-2s$, $x+d<L$.  Set $a=e^{-d/s}$ and $q=e^{x/s-m}$.  Thus
$a\geq e^{-1}$ and $q\leq e^{-2}$.  Since $(a-q)/(1-q)$ is increasing
in $a$ and decreasing in $q$,
\[
 \frac{\sqrt{\psi(x+d)}}{\sqrt{\psi(x)}}
 =\frac{a-q}{1-q}
 \geq\frac{e^{-1}-e^{-2}}{1-e^{-2}}
 =\frac1{e+1}.
\]
Squaring proves \textup{(ii)}.
For \textup{(iii)}, the monotonicity of $\psi$ and
$F([0,\infty))=N$ yield
\[
 \int_{(L-2s,\infty)}\psi(x)\,dF(x)
 \leq N\psi(L-2s)
 =N\bigl(e^{-(m-2)}-e^{-m}\bigr)^2
 =(e^2-1)^2Ne^{-2m}.
\]
Finally, suppose that $L\leq\varepsilon'$.  For $0\leq x\leq L$, set
$\theta=x/L\in[0,1]$.  Since $r\mapsto e^{-r}$ is convex,
\[
 e^{-x/s}=e^{-m\theta}
 \leq(1-\theta)+\theta e^{-m}.
\]
Hence
\[
 0\leq e^{-x/s}-e^{-m}
 \leq(1-\theta)(1-e^{-m})
 \leq1-\frac{x}{L}
 \leq1-\frac{x}{\varepsilon'}.
\]
After squaring, we obtain
$\psi(x)\leq(1-x/\varepsilon')^2$ on $[0,L]$.  For $x>L$,
$\psi(x)=0$, so $\psi\leq\varphi_{\varepsilon'}$ on
$[0,\infty)$.  This proves \textup{(iv)}.
\end{proof}
\subsection{One-scale collision estimate}
\begin{proposition}[One-scale collision estimate]\label{prop:generator}
Assume \cref{ass:CL}.  Let $s>0$, $m\geq4$, and $L=ms$, and suppose that
$L\leq\varepsilon\leq1/16$.  For
$F\in\mathcal B_1^+([0,\infty))$ of mass $N$ and energy $E$, set
$$U=\int_{[0,\infty)}\psi_{s,m}(x)\,dF(x),\qquad A=F([L,\varepsilon]).$$
Put $c_0=\sqrt{NE}$, then
\begin{equation}\label{eq:generator}
 \begin{aligned}
 &\iint\mathcal J[\psi_{s,m}]\,dF^{\otimes2}+\iiint\mathcal K[\psi_{s,m}]\,dF^{\otimes3}\\
 &\qquad\geq
 \frac{(e^{-1/2}-e^{-4})^2}{12\sqrt2}
 b_0\varepsilon^{\eta-1}A^2s^{3/2}-c_0U\\
&\qquad +\frac{b_0}{2(e+1)^2}\varepsilon^{\eta-1}A^2
 \frac{s}{\varepsilon+s}
 \bigl(U-(e^2-1)^2Ne^{-2m}\bigr)_+.
 \end{aligned}
\end{equation}
\end{proposition}
\begin{proof}
Write $\psi=\psi_{s,m}$.  We deal with the quadratic and cubic
contributions separately.

\medskip
\noindent\emph{Step 1: Quadratic gain from $[L,\varepsilon]$.}
Set
\(
 D=[L,\varepsilon]^2.
\)
By \eqref{eq:W-last-two-symmetry},
$\mathcal J[\psi](y,z)$ is symmetric in $y$ and $z$.  Thus, for
$(y,z)\in D$, it is without loss of generality to consider $y\leq z$.
Since $\operatorname{supp}\psi\subset[0,L]$ and $y,z\geq L$, we have
$\psi(y)=\psi(z)=0$.  Hence, for $0\leq x\leq y+z$,
\[
 \Delta\psi(x,y,z)
 =\psi(x)+\psi(y+z-x)
 \geq\psi(x)\geq 0.
\]
Since $W\geq0$, discarding the contribution from
$x\in(s/2,y+z]$ gives
\[
 \mathcal J[\psi](y,z)
 \geq\frac12\int_0^{s/2}
 W(x,y,z)\psi(x)\sqrt{x}\,dx.
\]
For $0\leq x\leq s/2$, we have
$x\leq s/2<L\leq y\leq z\leq\varepsilon$.  Hence
\cref{lem:W-lower,lem:profile}(i) imply
\begin{align}
 \mathcal J[\psi](y,z)
 &\geq
 \frac12\int_0^{s/2}
 W(x,y,z)\psi(x)\sqrt{x}\,dx \notag\\
 &\geq
 \frac{(e^{-1/2}-e^{-4})^2}{4}\,
 b_0\varepsilon^{\eta-1}
 \int_0^{s/2}\sqrt{x}\,dx \notag\\
 &=
 \frac{(e^{-1/2}-e^{-4})^2}{12\sqrt2}\,
 b_0\varepsilon^{\eta-1}s^{3/2},
 \qquad (y,z)\in D.
 \label{eq:quadratic-pointwise}
\end{align}

\medskip
\noindent\emph{Step 2: Quadratic loss outside \(D=[L,\varepsilon]^2\).}
By \eqref{eq:J-well-defined-bound} and
\eqref{eq:extension-half-moment},
$\mathcal J[\psi]\in L^1(F^{\otimes2})$. Write
\[
 \iint\mathcal J[\psi]\,dF^{\otimes2}
 =
 \iint_D\mathcal J[\psi]\,dF^{\otimes2}
 +
 \iint_{D^c}\mathcal J[\psi]\,dF^{\otimes2}.
\]
Integrating \eqref{eq:quadratic-pointwise} over $D$ and using
$F^{\otimes2}(D)=F([L,\varepsilon])^2=A^2$, we obtain
\[
 \iint_D\mathcal J[\psi]\,dF^{\otimes2}
 \geq
 \frac{(e^{-1/2}-e^{-4})^2}{12\sqrt2}\,
 b_0\varepsilon^{\eta-1}A^2s^{3/2}.
\]
On $D^c$,
\eqref{eq:J-pointwise-lower} then gives
\begin{align}
 \iint_{D^c}\mathcal J[\psi]\,dF^{\otimes2}
 &\geq
 -\frac12\iint_{D^c}
 \bigl(\psi(y)\sqrt z+\psi(z)\sqrt y\bigr)
 \,dF(y)\,dF(z) \notag\\
 &\geq
 -\frac12\iint
 \bigl(\psi(y)\sqrt z+\psi(z)\sqrt y\bigr)
 \,dF(y)\,dF(z) \notag\\
 &=
 -\left(\int\psi\,dF\right)
  \left(\int\sqrt{x}\,dF(x)\right) \notag\\
 &\geq-c_0U,
 \label{eq:quadratic-remainder}
\end{align}
where we used
\eqref{eq:extension-half-moment} in the last line.
Combining the estimates on $D$ and $D^c$ yields
\begin{equation}\label{eq:quadratic-integrated}
 \iint\mathcal J[\psi]\,dF^{\otimes2}
 \geq
 \frac{(e^{-1/2}-e^{-4})^2}{12\sqrt2}\,
 b_0\varepsilon^{\eta-1}A^2s^{3/2}
 -c_0U.
\end{equation}

\medskip
\noindent\emph{Step 3: Partition of the shell for the cubic term.}
Choose a disjoint partition of \([L,\varepsilon]\) by
\[
 n=\max\left\{
 1,\left\lceil\frac{\varepsilon-L}{s}\right\rceil
 \right\},\qquad  [L,\varepsilon]=\bigcup_{i=1}^n I_i
\]
into Borel intervals of equal length, using half-open intervals except
for the last one.  When $\varepsilon=L$, take $I_1=\{L\}$.  Then
$\operatorname{diam}(I_i)\leq s$ for every $i$, and
\[
 n\leq1+\frac{\varepsilon-L}{s}
 \leq\frac{\varepsilon+s}{s},\qquad \sum_{i=1}^nF(I_i)=A.
\]
The Cauchy--Schwarz inequality gives
\begin{equation}\label{eq:cell-square-lower}
 \sum_{i=1}^nF(I_i)^2
 \geq\frac1n
 \left(\sum_{i=1}^nF(I_i)\right)^2
 \geq\frac{s}{\varepsilon+s}A^2.
\end{equation}

\medskip
\noindent\emph{Step 4: Pointwise lower bound for the cubic gain.}
Fix $i$ and consider
\[
 0\leq x\leq L-2s,\qquad
  y,z\in I_i,\qquad  y\leq z.
\]
Since $I_i\subset[L,\varepsilon]$ and
$\operatorname{diam}(I_i)\leq s$, we have
\(
 x<y\leq z\leq\varepsilon,
 0\leq z-y\leq s.
\)
Moreover,
\(
 y+z-x
 \geq2L-(L-2s)
 =L+2s>L,\) and \(
 z\geq L.
\)
Since \(\operatorname{supp} \psi \subset [0,L],\)
$\psi(y+z-x)=\psi(z)=0$.  Consequently,
\begin{equation}\label{eq:centered-profile-exact}
 \Delta_{\mathrm{sym}}\psi(x,y,z)
 =\psi(x+z-y).
\end{equation}
Applying \cref{lem:profile}(ii) with $d=z-y\in[0,s]$, we obtain
\[
 \Delta_{\mathrm{sym}}\psi(x,y,z)
 \geq\frac{1}{(e+1)^2}\psi(x).
\]
Since $0\leq x<y\leq z\leq\varepsilon$,
\eqref{eq:W-lower} then implies
\begin{equation}\label{eq:cubic-pointwise-explicit}
 W(x,y,z)\Delta_{\mathrm{sym}}\psi(x,y,z)
 \geq
 \frac{b_0}{2(e+1)^2}
 \varepsilon^{\eta-1}\psi(x).
\end{equation}

\medskip
\noindent\emph{Step 5: Integration of the cubic gain and the boundary error.}
By \eqref{eq:convex-positivity},
\begin{align}
 \iiint\mathcal K[\psi]\,dF^{\otimes3}
 &\geq
 \iiint_{\{0\leq x<y\leq z\}}
 \chi_{y,z}W(x,y,z)
 \Delta_{\mathrm{sym}}\psi(x,y,z)\,dF^{\otimes3}
 \notag\\
 &\geq
 \sum_{i=1}^n
 \int_{\substack{0\leq x\leq L-2s\\
                  y,z\in I_i,\ y\leq z}}
 \chi_{y,z}W(x,y,z)
 \Delta_{\mathrm{sym}}\psi(x,y,z)
 \,dF(x)\,dF(y)\,dF(z)
 \notag\\
 &\geq
 \frac{b_0}{2(e+1)^2}\varepsilon^{\eta-1}
 \left(\int_{[0,L-2s]}\psi(x)\,dF(x)\right)
 \sum_{i=1}^n
 \int_{\substack{(y,z)\in I_i^2\\y\leq z}}
 \chi_{y,z}\,dF(y)\,dF(z)
 \notag\\
 &=
 \frac{b_0}{2(e+1)^2}\varepsilon^{\eta-1}
 \left(\int_{[0,L-2s]}\psi(x)\,dF(x)\right)
 \sum_{i=1}^nF(I_i)^2
 \notag\\
 &\geq
 \frac{b_0}{2(e+1)^2}\varepsilon^{\eta-1}
 \frac{s}{\varepsilon+s}A^2
 \int_{[0,L-2s]}\psi(x)\,dF(x).
 \label{eq:cubic-before-boundary}
\end{align}
It remains to compare the last cutoff integral with $U$.  Since
$\psi$ is supported in $[0,L]$,
\[
 \int_{[0,L-2s]}\psi(x)\,dF(x)
 =
 U-\int_{(L-2s,\infty)}\psi(x)\,dF(x).
\]
By \cref{lem:profile}(iii),
\[
 \int_{[0,L-2s]}\psi(x)\,dF(x)
 \geq U-(e^2-1)^2Ne^{-2m}.
\]
The integral on the left is nonnegative, and hence
\begin{equation}\label{eq:boundary-positive-part}
 \int_{[0,L-2s]}\psi(x)\,dF(x)
 \geq
 \bigl(U-(e^2-1)^2Ne^{-2m}\bigr)_+.
\end{equation}
Substituting \eqref{eq:boundary-positive-part} into
\eqref{eq:cubic-before-boundary} gives
\[
 \iiint\mathcal K[\psi]\,dF^{\otimes3}
 \geq
 \frac{b_0}{2(e+1)^2}\varepsilon^{\eta-1}A^2
 \frac{s}{\varepsilon+s}
 \bigl(U-(e^2-1)^2Ne^{-2m}\bigr)_+.
\]
Adding this inequality to \eqref{eq:quadratic-integrated} proves
\eqref{eq:generator}.
\end{proof}
\section{Multiscale weighted estimate}\label{sec:multiscale}
We now pass from the one-scale estimate in \cref{sec:one-scale} to the
simultaneous treatment of infinitely many low-energy scales.  This is the
main step toward producing an atom at zero energy.  We consider the dyadic
scales
\[
 \varepsilon_k=2^{-k-4},
 \qquad
 S_k=(\varepsilon_{k+1},\varepsilon_k],
 \qquad
 A_k(t)=F_t(S_k),
\]
where $A_k(t)$ is the mass in the $k$-th shell.  At each scale, we choose
$m_k$ and $s_k$ so that $L_k=m_ks_k=\varepsilon_{k+2}$, and define
\[
 U_k(t)=\int_{[0,\infty)}
 \psi_{s_k,m_k}(x)\,dF_t(x),
\]
where \(\operatorname{supp}\psi_{s_k,m_k}\subset [0,\varepsilon_{k+2}].\)
Since
$S_k\subset[\varepsilon_{k+2},\varepsilon_k]$, applying
\cref{prop:generator} with $L=\varepsilon_{k+2}$ and
$\varepsilon=\varepsilon_k$ gives, after absorbing the boundary error
for all sufficiently large $k$,
\[
 U_k'(t)+c_0U_k(t)
 \geq
 \gamma_kA_k(t)^2
 \left(U_k(t)+\frac{Q_k}{2}\right),
 \qquad c_0=\sqrt{NE}.
\]
Thus the mass in $S_k$ drives the transfer of particles into the region
below $\varepsilon_{k+2}$.  The term containing $Q_k$ initiates this
transfer, while the term containing $U_k(t)$ amplifies it through Bose
enhancement.
Fix a reference time $t_0\geq0$.  To measure the progress of the
transfer from the shell $S_k$ to energies below
$\varepsilon_{k+2}$, we first use the integrating factor
$e^{c_0(t-t_0)}$ to remove the linear loss in
\eqref{eq:relative-Uk-growth}.  Setting
$V_k(t)=e^{c_0(t-t_0)}U_k(t)$, we obtain
\[
 V_k'(t)
 \geq
 \gamma_kA_k(t)^2
 \left(V_k(t)+\frac{Q_k}{2}\right),
 \qquad t\geq t_0.
\]
We therefore introduce the
logarithmic functional
\begin{equation}\label{eq:Z-def-preview}
 Z_k(t)
 =
 \log\left(
 1+\frac{2V_k(t)}{Q_k}
 \right)
 =
 \log\left(
 1+\frac{2e^{c_0(t-t_0)}U_k(t)}{Q_k}
 \right).
\end{equation}
Upon differentiation, the logarithm cancels the factor
$V_k+Q_k/2$ and gives
\[
 Z_k'(t)\geq\gamma_kA_k(t)^2.
\]
Thus $Z_k(t)$ is a nondecreasing quantity that measures the progress of
the transfer from $S_k$ to the region below $\varepsilon_{k+2}$.
To specify when this transfer has reached the level needed later, fix a
target value $h>0$ and define
\[
 \Lambda_k(h)
 =
 \log\left(1+\frac{2h}{Q_k}\right).
\]
By the definition of $Z_k$,
\[
 Z_k(t)\geq\Lambda_k(h)
 \quad\Longleftrightarrow\quad
 e^{c_0(t-t_0)}U_k(t)\geq h.
\]
Since
$\psi_{s_k,m_k}\leq\varphi_{\varepsilon_{k+2}}$, this implies
\[
 N_{0,2}(F_t,\varepsilon_{k+2})
 \geq e^{-c_0(t-t_0)}h.
\]
We say that scale $k$ is \emph{$h$-complete at time $t$} when this
threshold is reached. For a single scale,
\[
 p_k(t)
 =
 \min\left\{1,\frac{Z_k(t)}{\Lambda_k(h)}\right\}
\]
then records the fraction of the required transfer that has been completed.
It is less than one while scale $k$ is $h$-incomplete and equals one
after the target has been reached.  To combine the completion levels of
all scales $k\geq j$, we need to choose suitable weights
$\omega_k(h)$ and consider
\[
 \mathcal P(t)
 =
 \sum_{k=j}^{\infty}\omega_k(h)p_k(t),
 \qquad
 \Theta_j(h)
 =
 \sum_{k=j}^{\infty}\omega_k(h).
\]
The functional then satisfies
\[
 0\leq\mathcal P(t)\leq\Theta_j(h),
 \qquad
 \mathcal P'(t)
 \geq
 \sum_{\substack{k\geq j\\
                  Z_k(t)<\Lambda_k(h)}}
 \frac{A_k(t)^2}{\omega_k(h)}
\]
for almost every $t\geq t_0$.  Hence $\mathcal P$ has a fixed upper
bound, while every incomplete shell that still carries mass forces $\mathcal P$ to grow. In the following \cref{sec:atom-formation}, we use these two bounds in a contradiction
argument.  If the condensate mass stays below the target level, then a
fixed amount of positive-energy mass must lie in $h$-incomplete shells.
The derivative lower bound then forces $\mathcal P$ to increase beyond
its upper bound $\Theta_j(h)$ in finite time, and hence a zero-energy atom
must form.
Throughout this section, $F_t(dx)$ denotes a conservative isotropic measure
solution with conserved mass $N>0$ and energy $E>0$, and we set
$c_0=\sqrt{NE}$.
\begin{lemma}[Borel measurability of interval masses]
\label{lem:set-mass-measurable}
For every $0\leq a<b<\infty$, the map
\(
 t\mapsto F_t((a,b])
\)
is Borel measurable on $[0,\infty)$.
\end{lemma}
\begin{proof}
Fix $r\geq0$.  For $n\geq1$, define
\[
 \chi_{r,n}(x)
 =
 \bigl(1-n(x-r)_+\bigr)_+,
 \qquad x\geq0.
\]
Then $\chi_{r,n}\in C_b([0,\infty))$ and
\[
 \chi_{r,n}(x)\downarrow\mathbf 1_{[0,r]}(x)
 \qquad\text{as }n\to\infty.
\]
By \cref{lem:narrow-continuity},
\[
 t\longmapsto\int_{[0,\infty)}\chi_{r,n}(x)\,dF_t(x)
\]
is continuous for every $n$.  Since $0\leq\chi_{r,n}\leq1$, dominated
convergence gives
\[
 F_t([0,r])
 =
 \lim_{n\to\infty}
 \int_{[0,\infty)}\chi_{r,n}(x)\,dF_t(x).
\]
Thus $t\mapsto F_t([0,r])$ is Borel measurable as a pointwise limit of
continuous functions.
Finally, for $0\leq a<b$,
\[
 F_t((a,b])
 =
 F_t([0,b])-F_t([0,a]).
\]
Both terms on the right are Borel measurable in $t$, and hence so is
$t\mapsto F_t((a,b])$.
\end{proof}
\subsection{Dyadic low-energy scales
\texorpdfstring{\((\varepsilon_k)_{k\geq j}\)}
{epsilon(k), k >= j}
and logarithmic functional
\texorpdfstring{\(Z_k(t)\)}{Z(k,t)}}
For $k\geq0$, set
\[
 \varepsilon_k=2^{-k-4},
 \qquad
 S_k=(\varepsilon_{k+1},\varepsilon_k],
 \qquad
 m_k=4\left(1+\log\frac1{\varepsilon_k}\right),
 \qquad
 s_k=\frac{\varepsilon_k}{4m_k}.
\]
For $t\geq0$, define
\begin{equation}\label{eq:Uk-Ak}
 A_k(t)=F_t(S_k),
 \qquad
 U_k(t)=\int_{[0,\infty)}
 \psi_{s_k,m_k}(x)\,dF_t(x).
\end{equation}
Here $A_k(t)$ is the mass in the dyadic shell $S_k$, while $U_k(t)$
measures a weighted amount of mass at a lower energy scale.  Indeed, the
choice of $s_k$ gives
\[
 L_k:=m_ks_k=\frac{\varepsilon_k}{4}
 =\varepsilon_{k+2},
 \qquad
 e^{-2m_k}=e^{-8}\varepsilon_k^8.
\]
By \cref{lem:profile},
$\operatorname{supp}\psi_{s_k,m_k}\subset[0,L_k]$.  Since
$L_k=\varepsilon_{k+2}<\varepsilon_{k+1}$, \(\psi_{s_k,m_k}\)
is supported strictly below $S_k$.  Thus the pair $(A_k,U_k)$ records
the mass in $S_k$ and the low-energy mass produced below
$\varepsilon_{k+2}$ respectively.  Moreover, the boundary error in
\cref{prop:generator} is
\[
 (e^2-1)^2Ne^{-2m_k}
 =(e^2-1)^2Ne^{-8}\varepsilon_k^8,
\]
which will be absorbed for all sufficiently large $k$ in
\cref{lem:scales}.  Since $S_k$ is a half-open interval,
\cref{lem:set-mass-measurable} implies that $A_k$ is Borel measurable.
By \cref{lem:profile,lem:test-extension}, $U_k\in C^1([0,\infty))$.
The one-scale estimate contains terms from quadratic source and cubic Bose-enhancement.  To write both terms with the same prefactor, define
\begin{equation}\label{eq:gamma-Q}
 \gamma_k=
 \frac{b_0}{2(e+1)^2}\varepsilon_k^{\eta-1}
 \frac{s_k}{\varepsilon_k+s_k},
 \qquad
 Q_k=
 \frac{(e^{-1/2}-e^{-4})^2(e+1)^2}{6\sqrt2}
 (\varepsilon_k+s_k)\sqrt{s_k}.
\end{equation}
Then
\[
 \gamma_kQ_k
 =
 \frac{(e^{-1/2}-e^{-4})^2}{12\sqrt2}\,
 b_0\varepsilon_k^{\eta-1}s_k^{3/2},
\]
so $\gamma_kQ_k$ is the coefficient of the quadratic source in
\cref{prop:generator}, while $\gamma_k$ is the coefficient of the cubic
feedback.
We now apply \cref{prop:generator} with
$m=m_k$, $s=s_k$, $L=L_k=\varepsilon_{k+2}$, and
$\varepsilon=\varepsilon_k$. Since
$S_k\subset[L_k,\varepsilon_k]$,
$F_t([L_k,\varepsilon_k])\geq A_k(t)$.  Therefore,
\cref{lem:test-extension,prop:generator} and
\eqref{eq:gamma-Q} give
\begin{align}
 U_k'(t)+c_0U_k(t)
 &\geq
 \gamma_kF_t([L_k,\varepsilon_k])^2
 \left[
 Q_k+
 \bigl(U_k(t)-(e^2-1)^2Ne^{-2m_k}\bigr)_+
 \right]
 \notag\\
 &\geq
 \gamma_kA_k(t)^2
 \left[
 Q_k+
 \bigl(U_k(t)-(e^2-1)^2Ne^{-2m_k}\bigr)_+
 \right]
 \label{eq:Uk-generator}
\end{align}
for every $t\geq0$.
The next lemma shows that the boundary correction \((e^2-1)^2Ne^{-2m_k}\) can be absorbed at all sufficiently small energy scales.
\begin{lemma}
\label{lem:scales}
Uniformly in $k\geq0$,
\begin{equation}\label{eq:scale-estimates}
 m_k\asymp1+k,
 \qquad
 \gamma_k\asymp
 b_0\frac{\varepsilon_k^{\eta-1}}{1+k},
 \qquad
 Q_k\asymp
 \frac{\varepsilon_k^{3/2}}{\sqrt{1+k}}.
\end{equation}
Moreover, there exists an integer $k_0=k_0(N)\geq0$ such that
\begin{equation}\label{eq:boundary-error}
 (e^2-1)^2Ne^{-2m_k}\leq\frac{Q_k}{2}
\end{equation}
for every $k\geq k_0$.
\end{lemma}
\begin{proof}
Since $\varepsilon_k=2^{-k-4}$, by the definitions,
\[
 m_k=4\bigl(1+(k+4)\log2\bigr),
 \qquad
 s_k=\frac{\varepsilon_k}{4m_k},
 \qquad
 \frac{s_k}{\varepsilon_k+s_k}
 =\frac1{4m_k+1}.
\]
It follows that
\[
 \gamma_k
 =\frac{b_0\varepsilon_k^{\eta-1}}
 {2(e+1)^2(4m_k+1)}
\]
and
\[
 Q_k=
 \frac{(e^{-1/2}-e^{-4})^2(e+1)^2}{12\sqrt2}
 \varepsilon_k^{3/2}
 \frac{1+(4m_k)^{-1}}{\sqrt{m_k}}.
\]
Because $m_k\asymp1+k$, these identities imply all three estimates in
\eqref{eq:scale-estimates}.
It remains to compare $Q_k$ with the boundary error.  Since
$e^{-2m_k}=e^{-8}\varepsilon_k^8$ and
$Q_k\asymp\varepsilon_k^{3/2}/\sqrt{m_k}$, we have
\[
 \frac{(e^2-1)^2Ne^{-2m_k}}{Q_k}
 \lesssim_N
 \sqrt{m_k}\,\varepsilon_k^{13/2}
 \lesssim_N
 \sqrt{1+k}\,2^{-13k/2}
 \to 0.
\]
Hence the ratio is at most $1/2$ for every sufficiently large $k$.
Choosing $k_0=k_0(N)$ accordingly proves
\eqref{eq:boundary-error}.
\end{proof}
Fix $k_0=k_0(N)$ as in \cref{lem:scales}.  For $k\geq k_0$, set
\(
 r_k=(e^2-1)^2Ne^{-2m_k}.
\)
By \eqref{eq:boundary-error}, $0\leq r_k\leq Q_k/2$.  Since
$(u-r_k)_+\geq u-r_k$ for every $u\geq0$, we have
\[
 Q_k+(u-r_k)_+
 \geq u+Q_k-r_k
 \geq u+\frac{Q_k}{2}.
\]
Applying this inequality with $u=U_k(t)$ in
\eqref{eq:Uk-generator} gives
\begin{equation}\label{eq:relative-Uk-growth}
 U_k'(t)+c_0U_k(t)
 \geq
 \gamma_kA_k(t)^2
 \left(U_k(t)+\frac{Q_k}{2}\right),
 \qquad t\geq0.
\end{equation}
Fix a reference time $t_0\geq0$.  Set
\[
 V_k(t)=e^{c_0(t-t_0)}U_k(t),
 \qquad t\geq t_0.
\]
The factor $e^{c_0(t-t_0)}$ absorbs the linear term $c_0U_k$ in
\eqref{eq:relative-Uk-growth}.  Indeed,
\begin{align}
 V_k'(t)
 &=e^{c_0(t-t_0)}
   \bigl(U_k'(t)+c_0U_k(t)\bigr) \notag\\
 &\geq
 \gamma_kA_k(t)^2
 \left(
 V_k(t)+\frac{e^{c_0(t-t_0)}Q_k}{2}
 \right) \notag\\
 &\geq
 \gamma_kA_k(t)^2
 \left(V_k(t)+\frac{Q_k}{2}\right),
 \qquad t\geq t_0.
 \label{eq:Vk-growth}
\end{align}
We then introduce the logarithmic functional as
\begin{equation}\label{eq:Z-def}
 Z_k(t)
 =
 \log\left(
 \frac{V_k(t)+Q_k/2}{Q_k/2}
 \right)
 =
 \log\left(
 1+\frac{2e^{c_0(t-t_0)}U_k(t)}{Q_k}
 \right),
 \qquad t\geq t_0.
\end{equation}
We suppress the dependence of $V_k$ and $Z_k$ on the reference time
$t_0$ in the notation.
\begin{lemma}[Growth of \(Z_k\)]
\label{lem:log-growth}
For every $t_0\geq0$ and $k\geq k_0$, the function
$Z_k$ belongs to $C^1([t_0,\infty))$ and satisfies
\begin{equation}\label{eq:Z-growth}
 Z_k'(t)\geq\gamma_kA_k(t)^2,
 \qquad t\geq t_0.
\end{equation}
In particular, $Z_k$ is nondecreasing.
\end{lemma}
\begin{proof}
By \cref{lem:test-extension}, $U_k\in C^1([0,\infty))$, and hence
$V_k,Z_k\in C^1([t_0,\infty))$.  From
\eqref{eq:Z-def} and \eqref{eq:Vk-growth},
\[
 Z_k'(t)
 =
 \frac{V_k'(t)}{V_k(t)+Q_k/2}
 \geq
 \gamma_kA_k(t)^2.
\]
This proves \eqref{eq:Z-growth}.  Since its right-hand side is
nonnegative, $Z_k$ is nondecreasing.
\end{proof}
Fix $t_0\geq0$ and $h>0$.  For the functions $Z_k$ defined relative to
this reference time \(t_0\), set
\begin{equation}\label{eq:Lambda-def}
 \Lambda_k(h):=
 \log\left(1+\frac{2h}{Q_k}\right)>0.
\end{equation}
By \eqref{eq:Z-def},
\begin{equation}\label{eq:level-equivalence}
 V_k(t)=e^{c_0(t-t_0)}U_k(t)\geq h
 \quad\Longleftrightarrow\quad
 Z_k(t)\geq\Lambda_k(h).
\end{equation}
\begin{definition}[Complete and incomplete scales]
\label{def:complete-scale}
Let $k\geq k_0$ and $t\geq t_0$.  We say that scale $k$ is
\emph{$h$-complete at time $t$} if
\(
 e^{c_0(t-t_0)}U_k(t)\geq h,
\)
or equivalently, if $Z_k(t)\geq\Lambda_k(h)$.  Otherwise, scale $k$ is
\emph{$h$-incomplete at time $t$}.
\end{definition}
Since $Z_k$ is nondecreasing, a scale that is $h$-complete at some time
remains $h$-complete at all later times.  Moreover, if scale $k$ is
$h$-complete at time $t$, then
\(
 U_k(t)\geq e^{-c_0(t-t_0)}h.
\)
The identity $m_ks_k=\varepsilon_{k+2}$ and
\cref{lem:profile}(iv) give
$\psi_{s_k,m_k}\leq\varphi_{\varepsilon_{k+2}}$.  Therefore
\begin{equation}\label{eq:level-cutoff-bound}
 N_{0,2}(F_t,\varepsilon_{k+2})
 \geq U_k(t)
 \geq e^{-c_0(t-t_0)}h
\end{equation}
whenever scale $k$ is $h$-complete at time $t$.
\subsection{Weighted multiscale functional
\texorpdfstring{\(\mathcal P\)}{P}}
Fix $h>0$, integers $j\geq k_0$, and $M\geq j$. The estimate \eqref{eq:Z-growth} controls the
transfer of mass at one scale at a time.  Once scale $k$ becomes
$h$-complete, the transfer from $S_k$ to energies below
$\varepsilon_{k+2}$ has reached the level required in the subsequent
condensate argument in \cref{sec:atom-formation}.  We therefore set
\[
 p_k(t)
 =
 \min\left\{1,\frac{Z_k(t)}{\Lambda_k(h)}\right\},
 \qquad j\leq k\leq M,
\]
which records the fraction of this transfer that has been completed:
$p_k(t)=Z_k(t)/\Lambda_k(h)$ while scale $k$ is $h$-incomplete, and
$p_k(t)=1$ once it becomes $h$-complete.
We first measure the completion of all scales from $j$ to $M$ simultaneously, and let \(M\to\infty\) later.
Choose positive weights $\omega_k>0$ and define the weighted
multiscale functional
\[
 \mathcal P_{\omega,M}(t)
 =
 \sum_{k=j}^{M}\omega_k p_k(t)
 =
 \sum_{k=j}^{M}\omega_k
 \min\left\{1,\frac{Z_k(t)}{\Lambda_k(h)}\right\}.
\]
Since $0\leq p_k(t)\leq1$,
\[
 0\leq\mathcal P_{\omega,M}(t)
 \leq\sum_{k=j}^{M}\omega_k.
\]
If scale $k$ is $h$-incomplete, then
$Z_k(t)<\Lambda_k(h)$. \cref{lem:log-growth} gives
\[
 \frac{d}{dt}
 \left(
  \omega_k\frac{Z_k(t)}{\Lambda_k(h)}
 \right)
 \geq
 \frac{\omega_k\gamma_k}{\Lambda_k(h)}A_k(t)^2.
\]
Consequently, almost everywhere,
\[
 \mathcal P_{\omega,M}'(t)
 \geq
 \sum_{\substack{j\leq k\leq M\\
                  Z_k(t)<\Lambda_k(h)}}
 \frac{\omega_k\gamma_k}{\Lambda_k(h)}A_k(t)^2.
\]
In \cref{sec:atom-formation}, we will consider the mass of \(h\)-incomplete shells to obtain condensation, as we introduced in the beginning of this section. To convert the derivative estimate above into a bound involving the mass
in incomplete shells, let $I\subset\{j,\ldots,M\}$ be a collection of
incomplete scales.  Cauchy--Schwarz then gives
\begin{equation}\label{eq:generic-weight-CS}
 \left(\sum_{k\in I}A_k(t)\right)^2
 \leq
 \left(
  \sum_{k\in I}
  \frac{\omega_k\gamma_k}{\Lambda_k(h)}A_k(t)^2
 \right)
 \left(
  \sum_{k\in I}
  \frac{\Lambda_k(h)}{\omega_k\gamma_k}
 \right)\leq \mathcal P_{\omega,M}'(t)\left(
  \sum_{k\in I}
  \frac{\Lambda_k(h)}{\omega_k\gamma_k}
 \right).
\end{equation}
The provisional weights \(\omega_k\) therefore produce two different sums in the subsequent arguments:
\[
 \sum_k\omega_k,
 \qquad
 \sum_k\frac{\Lambda_k(h)}{\omega_k\gamma_k}.
\]
The first bounds \(\mathcal P\) itself, while the second is the
Cauchy--Schwarz factor in \eqref{eq:generic-weight-CS}.  Both of them must remain
finite on arbitrarily fine tails to make sure that the contradiction argument for condensation in \cref{sec:atom-formation} works. The optimal balance between these two requirements follows from
\begin{equation}\label{eq:weight-optimality}
 \left(\sum_{k=j}^{M}\omega_k\right)
 \left(
  \sum_{k=j}^{M}
  \frac{\Lambda_k(h)}{\omega_k\gamma_k}
 \right)
 \geq
 \left(
  \sum_{k=j}^{M}
  \sqrt{\frac{\Lambda_k(h)}{\gamma_k}}
 \right)^2.
\end{equation}
Equality holds precisely when
\(
 \omega_k
 =c\sqrt{{\Lambda_k(h)}/{\gamma_k}}
\)
with the same constant $c>0$ at every scale.  We take $c=1$ and define
\begin{equation}\label{eq:weights}
 \omega_k(h):=
 \sqrt{\frac{\Lambda_k(h)}{\gamma_k}},
 \qquad
 \Theta_j(h):=
 \sum_{k=j}^{\infty}\omega_k(h),
 \qquad k,j\geq k_0.
\end{equation}
\begin{lemma}[Summability of \(\Theta_j(h)\)]
\label{lem:weight-sum}
For every fixed $h>0$,
\begin{equation}\label{eq:omega-asymp}
 \omega_k(h)
 \asymp_{h,\eta}
 b_0^{-1/2}(1+k)
 \varepsilon_k^{(1-\eta)/2}
 \qquad \text{as }k\to\infty.
\end{equation}
Consequently, if $\eta<1$, then for every $j\geq k_0$,
\begin{equation}\label{eq:Theta-finite}
 0<\Theta_j(h)<\infty,
 \qquad
 \Theta_j(h)
 \lesssim_{h,\eta}
 b_0^{-1/2}(1+j)
 \varepsilon_j^{(1-\eta)/2}
 \to 0
 \quad \text{as }j\to\infty.
\end{equation}
\end{lemma}
\begin{proof}
By \eqref{eq:scale-estimates},
\[
 \log\frac1{Q_k}
 =
 \frac32(k+4)\log2+\frac12\log(1+k)+O(1)
 \asymp 1+k.
\]
In particular, $Q_k\to0$.  Fix $h>0$.  For all sufficiently large
$k$, we have $2h/Q_k\geq1$, and hence
\[
 \log\frac{2h}{Q_k}
 \leq
 \Lambda_k(h)
 =
 \log\left(1+\frac{2h}{Q_k}\right)
 \leq
 \log\frac{4h}{Q_k}.
\]
It follows that
\[
 \Lambda_k(h)
 =
 \log\frac1{Q_k}+O_h(1)
 \asymp_h 1+k.
\]
Using the definition of $\omega_k(h)$ and the estimate for $\gamma_k$,
we obtain
\[
 \omega_k(h)^2
 =
 \frac{\Lambda_k(h)}{\gamma_k}
 \asymp_{h,\eta}
 b_0^{-1}(1+k)^2\varepsilon_k^{1-\eta}.
\]
Taking square roots proves \eqref{eq:omega-asymp}. Since $\eta<1$, setting
\(
 \rho=2^{-(1-\eta)/2}\in(0,1),
\)
we have
\(
 \varepsilon_k^{(1-\eta)/2}=\rho^{k+4}.
\)
For all sufficiently large $j$, \eqref{eq:omega-asymp} gives
\begin{align*}
 \Theta_j(h)
 &\lesssim_{h,\eta}
 b_0^{-1/2}
 \sum_{k=j}^{\infty}(1+k)\rho^{k+4}\\
 &=
 b_0^{-1/2}\rho^{j+4}
 \sum_{\ell=0}^{\infty}(1+j+\ell)\rho^\ell\\
 &=
 b_0^{-1/2}\rho^{j+4}
 \left(
 \frac{1+j}{1-\rho}
 +\frac{\rho}{(1-\rho)^2}
 \right)\\
 &\lesssim_{h,\eta}
 b_0^{-1/2}(1+j)\varepsilon_j^{(1-\eta)/2}.
\end{align*}
Hence $\Theta_j(h)<\infty$ for all sufficiently large $j$.
For the finitely many indices $k_0\leq j$ preceding this range, the same
bound follows after increasing the constant.  Since every
$\omega_k(h)$ is positive, $\Theta_j(h)>0$.  Finally,
\[
 (1+j)\varepsilon_j^{(1-\eta)/2}
 =
 (1+j)2^{-(j+4)(1-\eta)/2}
 \to 0,
\]
which proves \eqref{eq:Theta-finite}.
\end{proof}
We now pass to the limit \(M\to\infty\).  Fix $t_0\geq0$, $h>0$, and
$j\geq k_0$, and define
\begin{equation}\label{eq:Pcal}
 \mathcal P(t)
 =
 \sum_{k=j}^{\infty}\omega_k(h)
 \min\left\{1,\frac{Z_k(t)}{\Lambda_k(h)}\right\},
 \qquad t\geq t_0.
\end{equation}
\begin{lemma}[Estimate of \(\mathcal P\) and \(\mathcal P'\)]
\label{lem:truncated-sum}
The series in \eqref{eq:Pcal} converges uniformly on
$[t_0,\infty)$ and defines a bounded, nondecreasing, locally absolutely
continuous function.  It satisfies
\begin{equation}\label{eq:P-upper-global}
 0\leq\mathcal P(t)\leq\Theta_j(h),
 \qquad t\geq t_0,
\end{equation}
and, for almost every $t\geq t_0$,
\begin{equation}\label{eq:truncated-sum-derivative}
 \mathcal P'(t)
 =
 \sum_{\substack{k\geq j\\
                  Z_k(t)<\Lambda_k(h)}}
 \frac{\omega_k(h)}{\Lambda_k(h)}Z_k'(t)\geq \sum_{\substack{k\geq j\\
                  Z_k(t)<\Lambda_k(h)}}
 \frac{A_k(t)^2}{\omega_k(h)}.
\end{equation}
\end{lemma}
\begin{proof}
For brevity, write
\(
 \omega_k=\omega_k(h),
 \Lambda_k=\Lambda_k(h),
\)
and set
\[
 p_k(t)
 =
 \min\left\{1,\frac{Z_k(t)}{\Lambda_k}\right\},
 \qquad k\geq j.
\]
By \cref{lem:log-growth}, $Z_k\in C^1([t_0,\infty))$ is
nondecreasing.  Hence $p_k$ is nondecreasing and locally absolutely
continuous.  Moreover, chain rule implies
\begin{equation}\label{eq:truncation-chain-rule}
 p_k'(t)
 =
 \frac{Z_k'(t)}{\Lambda_k}
 \mathbf 1_{\{Z_k(t)<\Lambda_k\}}
\end{equation}
for almost every $t\geq t_0$. By \cref{lem:weight-sum},
$\displaystyle\sum_{k=j}^\infty\omega_k=\Theta_j(h)<\infty$.  Since
$0\leq p_k(t)\leq1$,
the Weierstrass test therefore gives uniform convergence of
\eqref{eq:Pcal} on $[t_0,\infty)$ and
\(
 0\leq\mathcal P(t)\leq\Theta_j(h).
\)
Since every $p_k$ is nondecreasing, $\mathcal P$ is also
nondecreasing.
We next prove local absolute continuity and identify the derivative for \(\mathcal P\).
Fix $t_0\leq a<b<\infty$ and define
\[
 g_k(t)
 =
 \frac{Z_k'(t)}{\Lambda_k}
 \mathbf 1_{\{Z_k(t)<\Lambda_k\}},
 \qquad
 D(t)=\sum_{k=j}^{\infty}\omega_kg_k(t).
\]
By \eqref{eq:truncation-chain-rule}, $p_k'=g_k$ almost everywhere.
Since $g_k$ are nonnegative, Tonelli's theorem gives
\begin{align*}
 \int_a^b D(t)\,dt=
 \sum_{k=j}^{\infty}\omega_k
 \int_a^b g_k(t)\,dt=
 \sum_{k=j}^{\infty}\omega_k
 (p_k(b)-p_k(a))
 \leq
 \sum_{k=j}^{\infty}\omega_k
 =
 \Theta_j(h).
\end{align*}
Thus $D\in L^1(a,b)$.  For every $t\in[a,b]$, another application of
Tonelli's theorem yields
\begin{align*}
 \mathcal P(t)-\mathcal P(a)=
 \sum_{k=j}^{\infty}\omega_k
 (p_k(t)-p_k(a))=
 \sum_{k=j}^{\infty}\omega_k
 \int_a^t g_k(\tau)\,d\tau
 =
 \int_a^tD(\tau)\,d\tau.
\end{align*}
Hence $\mathcal P$ is absolutely continuous on $[a,b]$ and
$\mathcal P'(t)=D(t)$ for almost every $t\in[a,b]$.  Since $[a,b]$
was arbitrary,
\[
 \mathcal P'(t)
 =
 \sum_{\substack{k\geq j\\
                  Z_k(t)<\Lambda_k(h)}}
 \frac{\omega_k(h)}{\Lambda_k(h)}Z_k'(t)
\]
for almost every $t\geq t_0$.
Finally, if $Z_k(t)<\Lambda_k(h)$, then
\eqref{eq:Z-growth} gives
\[
 \frac{\omega_k(h)}{\Lambda_k(h)}Z_k'(t)
 \geq
 \frac{\omega_k(h)\gamma_k}{\Lambda_k(h)}A_k(t)^2
 =
 \frac{A_k(t)^2}{\omega_k(h)}.
\]
Summing over all $h$-incomplete scales proves
\eqref{eq:truncated-sum-derivative}.
\end{proof}
\section{Finite-time condensation and convergence to equilibrium}
\label{sec:atom-formation}
We now use the estimate of \(\mathcal P\) and \(\mathcal P'\) in
\cref{lem:truncated-sum} to prove finite-time condensation.  In the
first subsection, we show that
$N_{0,2}(F_{t_0},\varepsilon_j)\geq h$ at a sufficiently fine energy
scale forces a zero-energy atom to form in finite time.  Fix
$\delta\in(0,1/2)$ and consider the interval
$I=[t_0,t_0+\tau_j(h,\delta)]$, where $j$ is chosen so that
$\tau_j(h,\delta)\leq1$.  The proof is by contradiction.  Suppose that
\[
 F_t(\{0\})
 <
 (1-\delta)e^{-c_0(t-t_0)}h
 \qquad\text{for every }t\in I.
\]
For each $t\in I$, we choose an index $\nu(t)$ such that every scale
$k\geq\nu(t)$ is $h$-incomplete, while
$N_{0,2}(F_t,\varepsilon_{\nu(t)})
\geq e^{-c_0(t-t_0)}h$.  Since
\[
 N_{0,2}(F_t,\varepsilon_{\nu(t)})
 \leq
 F_t(\{0\})+F_t((0,\varepsilon_{\nu(t)}]),
\]
the assumed upper bound on $F_t(\{0\})$ gives
\[
 F_t((0,\varepsilon_{\nu(t)}])
 =
 \sum_{k\geq\nu(t)}A_k(t)
 \geq
 \delta e^{-c_0(t-t_0)}h.
\]
Thus, at every time $t\in I$, a fixed amount of mass lies in
the $h$-incomplete shells.  The derivative estimate
\eqref{eq:truncated-sum-derivative} then gives a uniform positive lower bound for
$\mathcal P'(t)$ on $I$. Integrating this bound forces $\mathcal P$ to
increase by more than its upper bound $\Theta_j(h)$, contradicting
\eqref{eq:P-upper-global}.  Hence a condensate forms within $I$, and
\eqref{eq:atom-persistence} keeps it positive at all later times.
In the second subsection, we obtain the required assumption $N_{0,2}(F_{t_0},\varepsilon_j)\geq h$
from semi-strong relaxation to the Bose--Einstein equilibrium when
$N_c>0$, equivalently $\overline T/\overline T_c<1$.  Choosing $j$ sufficiently large and
applying the criterion from the first subsection gives finite-time
condensation.  Applying the same criterion with
$t_0=t-\tau_j(h,\delta)$ for every sufficiently large $t$ then yields $F_t(\{0\})\to N_c$.  The Cai--Lu strong-convergence
criterion \cite{CaiLu2026} then upgrades the semi-strong convergence to
$\|F_t-F_{\rm be}\|_1\to0$.
\subsection{\texorpdfstring
{\(N_{0,2}(F_{t_0},\varepsilon_j)\geq h\)}
{N02(F(t0), epsilon(j)) >= h}
implies condensation}
\begin{theorem}
\label{thm:low-energy-to-condensate}
Assume \cref{ass:CL}, and let $F_t$ be a conservative isotropic measure
solution with conserved mass $N>0$ and energy $E>0$; set
$c_0=\sqrt{NE}$.  Fix $t_0\geq0$, $h\in(0,N]$,
$\delta\in(0,1/2)$, and an integer $j\geq k_0$, where
$k_0=k_0(N)$ is supplied by \cref{lem:scales}.  Suppose
$N_{0,2}(F_{t_0},\varepsilon_j)\geq h$.
Set
\begin{equation}\label{eq:formation-interval}
 \tau_j(h,\delta)
 =2e^{2c_0}\delta^{-2}h^{-2}\Theta_j(h)^2.
\end{equation}
If $\tau_j(h,\delta)\leq1$, then there exists
$T\in[t_0,t_0+\tau_j(h,\delta)]$ such that
\begin{equation}\label{eq:atom-lower-bound}
 F_T(\{0\})
 \geq(1-\delta)e^{-c_0(T-t_0)}h>0.
\end{equation}
Moreover, every $t\geq T$ satisfies
$F_t(\{0\})\geq(1-\delta)e^{-c_0(t-t_0)}h$.
The interval length also satisfies
\begin{equation}\label{eq:formation-interval-rate}
 \tau_j(h,\delta)
 \lesssim_{h,\delta,N,E,b_0,\eta}
 (1+j)^2\varepsilon_j^{1-\eta}\to 0.
\end{equation}
\end{theorem}
\begin{proof}
Set
\(
 \tau=\tau_j(h,\delta),
 I=[t_0,t_0+\tau],
\)
and let $\mathcal P$ be the multiscale functional
\eqref{eq:Pcal} associated with $t_0$, $h$, and $j$.
Assume for contradiction that
\begin{equation}\label{eq:contrary-atom-bound}
 F_t(\{0\})
 <
 (1-\delta)e^{-c_0(t-t_0)}h
 \qquad\text{for every }t\in I.
\end{equation}

\medskip
\noindent\emph{Step 1: The set of complete scales is finite.}
For each $t\in I$, define
\[
 \mathcal C_h(t)
 =
 \{k\geq j:Z_k(t)\geq\Lambda_k(h)\}.
\]
We claim that $\mathcal C_h(t)$ is finite.  Otherwise, there is a
sequence $k_n\in\mathcal C_h(t)$ such that $k_n\to\infty$.  Since each
$k_n$ is $h$-complete, \eqref{eq:level-cutoff-bound} gives
\[
 N_{0,2}(F_t,\varepsilon_{k_n+2})
 \geq e^{-c_0(t-t_0)}h.
\]
Letting $n\to\infty$ and using
$\varepsilon_{k_n+2}\to0$ and \eqref{eq:cutoff-limit}, we obtain
\[
 F_t(\{0\})
 \geq e^{-c_0(t-t_0)}h,
\]
which contradicts \eqref{eq:contrary-atom-bound}.  Thus
$\mathcal C_h(t)$ is finite.

\medskip
\noindent\emph{Step 2: Construction of an incomplete tail.}
Define
\begin{equation}\label{eq:nu-definition}
 \nu(t)=
 \begin{cases}
 j,
 &\mathcal C_h(t)=\varnothing,\\
 2+\max\mathcal C_h(t),
 &\mathcal C_h(t)\neq\varnothing.
 \end{cases}
\end{equation}
We first show that
\begin{equation}\label{eq:nu-cutoff-bound}
 N_{0,2}(F_t,\varepsilon_{\nu(t)})
 \geq e^{-c_0(t-t_0)}h.
\end{equation}
If $\mathcal C_h(t)=\varnothing$, then $\nu(t)=j$, and
\eqref{eq:cutoff-persistence} together with the assumption
$N_{0,2}(F_{t_0},\varepsilon_j)\geq h$ gives
\[
 N_{0,2}(F_t,\varepsilon_{\nu(t)})
 =
 N_{0,2}(F_t,\varepsilon_j)
 \geq e^{-c_0(t-t_0)}h.
\]
If $\mathcal C_h(t)\neq\varnothing$, let
$k_*=\max\mathcal C_h(t)$.  Then $k_*$ is $h$-complete and
$\nu(t)=k_*+2$, so \eqref{eq:level-cutoff-bound} again gives
\eqref{eq:nu-cutoff-bound}.
By the definition of $\nu(t)$, every scale $k\geq\nu(t)$ is
$h$-incomplete.  Moreover,
\begin{equation}\label{eq:shell-partition}
 \bigcup_{k=\nu(t)}^\infty S_k
 =
 (0,\varepsilon_{\nu(t)}].
\end{equation}

\medskip
\noindent\emph{Step 3: Positive mass in the incomplete shells.}
For every $\varepsilon>0$,
\begin{equation}\label{eq:cutoff-local-mass}
 \begin{aligned}
 N_{0,2}(F_t,\varepsilon)
 &=
 F_t(\{0\})
 +
 \int_{(0,\varepsilon]}
 \left(1-\frac{x}{\varepsilon}\right)^2\,dF_t(x)\\
 &\leq
 F_t(\{0\})+F_t((0,\varepsilon]).
 \end{aligned}
\end{equation}
Using \eqref{eq:shell-partition}, we have
\[
 F_t((0,\varepsilon_{\nu(t)}])
 =
 \sum_{k\geq\nu(t)}A_k(t).
\]
Therefore, by \eqref{eq:nu-cutoff-bound},
\eqref{eq:cutoff-local-mass}, and
\eqref{eq:contrary-atom-bound},
\begin{align}
 \sum_{k\geq\nu(t)}A_k(t)\geq
 N_{0,2}(F_t,\varepsilon_{\nu(t)})
 -F_t(\{0\}) \geq
 \delta e^{-c_0(t-t_0)}h
 \geq
 \delta e^{-c_0}h.
 \label{eq:positive-energy-tail-mass}
\end{align}
Here the last inequality follows from
$0\leq t-t_0\leq\tau\leq1$.

\medskip
\noindent\emph{Step 4: Growth of \(\mathcal P\).}
By \cref{lem:truncated-sum},
\begin{equation}\label{eq:P-bound-on-I}
 0\leq\mathcal P(t)\leq\Theta_j(h),
 \qquad t\in I,
\end{equation}
and, for almost every $t\in I$,
\[
 \mathcal P'(t)
 \geq
 \sum_{\substack{k\geq j\\
                  Z_k(t)<\Lambda_k(h)}}
 \frac{A_k(t)^2}{\omega_k(h)}\geq \sum_{k\geq\nu(t)}
 \frac{A_k(t)^2}{\omega_k(h)},
\]
because every $k\geq\nu(t)$ is $h$-incomplete. For every $M\geq\nu(t)$, Cauchy--Schwarz gives
\[
 \left(
 \sum_{k=\nu(t)}^M A_k(t)
 \right)^2
 \leq
 \left(
 \sum_{k=\nu(t)}^M
 \frac{A_k(t)^2}{\omega_k(h)}
 \right)
 \left(
 \sum_{k=\nu(t)}^M\omega_k(h)
 \right).
\]
Letting $M\to\infty$, we obtain
\[
 \sum_{k\geq\nu(t)}
 \frac{A_k(t)^2}{\omega_k(h)}
 \geq
 {
 \left(\displaystyle\sum_{k\geq\nu(t)}A_k(t)\right)^2
 }\left({
 \displaystyle\sum_{k\geq\nu(t)}\omega_k(h)
 }\right)^{-1}.
\]
Since $\nu(t)\geq j$,
\[
 \sum_{k\geq\nu(t)}\omega_k(h)
 \leq\Theta_j(h).
\]
Combining these estimates with
\eqref{eq:positive-energy-tail-mass}, we have
\begin{equation}\label{eq:P-uniform-derivative}
 \mathcal P'(t)
 \geq
 \frac{\delta^2e^{-2c_0}h^2}{\Theta_j(h)}
\end{equation}
for almost every $t\in I$.

\medskip
\noindent\emph{Step 5: Contradiction, persistence, and the time estimate.}
Since $\mathcal P$ is absolutely continuous by
\cref{lem:truncated-sum},  integrating
\eqref{eq:P-uniform-derivative} over $I$ and using
\eqref{eq:formation-interval}, we obtain
\begin{align*}
 \mathcal P(t_0+\tau)-\mathcal P(t_0)
\geq
 \frac{\delta^2e^{-2c_0}h^2}{\Theta_j(h)}\,\tau=
 2\Theta_j(h).
\end{align*}
On the other hand, \eqref{eq:P-bound-on-I} gives
\(
 \mathcal P(t_0+\tau)-\mathcal P(t_0)
 \leq\Theta_j(h).
\)
This contradiction shows that
\eqref{eq:contrary-atom-bound} cannot hold throughout $I$.  Hence there
exists $T\in[t_0,t_0+\tau]$ such that
\[
 F_T(\{0\})
 \geq
 (1-\delta)e^{-c_0(T-t_0)}h>0,
\]
which proves \eqref{eq:atom-lower-bound}.
For every $t\geq T$,
\eqref{eq:atom-persistence} gives
\[
 \begin{aligned}
 F_t(\{0\})\geq e^{-c_0(t-T)}F_T(\{0\})\geq
 (1-\delta)e^{-c_0(t-T)}
 e^{-c_0(T-t_0)}h=
 (1-\delta)e^{-c_0(t-t_0)}h.
 \end{aligned}
\]
Finally, \eqref{eq:Theta-finite} and
\eqref{eq:formation-interval} yield
\[
 \tau_j(h,\delta)
 =
 2e^{2c_0}\delta^{-2}h^{-2}\Theta_j(h)^2
 \lesssim_{h,\delta,N,E,b_0,\eta}
 (1+j)^2\varepsilon_j^{1-\eta}.
\]
Since $\eta<1$, the right-hand side tends to zero as $j\to\infty$.
This proves \eqref{eq:formation-interval-rate}.
\end{proof}
\subsection{Semi-strong relaxation implies condensation and convergence}
\begin{lemma}
\label{lem:late-cutoff}
Let $F_t$ be a conservative isotropic measure solution whose mass and
energy agree with those of $F_{\mathrm{be}}$.  Assume
$\|F_t-F_{\mathrm{be}}\|_1^\circ\to0$ as $t\to\infty$ and $N_c>0$.
Fix $h\in(0,N_c)$ and
$j\geq0$.  Then there exists $t_*=t_*(h,j)\geq0$ such that
\begin{equation}\label{eq:late-cutoff}
 N_{0,2}(F_{t_0},\varepsilon_j)\geq h
 \qquad\text{for every }t_0\geq t_*.
\end{equation}
If for some $\lambda>0$ and $C_\lambda>0$,
$\|F_t-F_{\mathrm{be}}\|_1^\circ\leq
C_\lambda(1+t)^{-\lambda/2}$ for every $t\geq0$, it is enough to choose
$t_*$ so that
\begin{equation}\label{eq:late-cutoff-time}
 C_\lambda(1+t_*)^{-\lambda/2}
 \leq\frac{\varepsilon_j}{2}(N_c-h).
\end{equation}
\end{lemma}
\begin{proof}
Since $\varphi_{\varepsilon_j}(0)=1$ and
$\varphi_{\varepsilon_j}\geq0$, we have
\[
 N_{0,2}(F_{\mathrm{be}},\varepsilon_j)
 \geq F_{\mathrm{be}}(\{0\})
 =N_c.
\]
Since $F_{t_0}$ and $F_{\mathrm{be}}$ have the same total mass, \cref{lem:cutoff-comparison} gives, for every $t_0\geq0$,
\begin{equation}\label{eq:late-cutoff-basic}
 N_{0,2}(F_{t_0},\varepsilon_j)
 \geq
 N_c-\frac{2}{\varepsilon_j}
 \norm{F_{t_0}-F_{\mathrm{be}}}_1^\circ.
\end{equation}
Since
$\norm{F_t-F_{\mathrm{be}}}_1^\circ\to0$ as $t\to\infty$, there exists
$t_*=t_*(h,j)$ such that
\[
 \norm{F_{t_0}-F_{\mathrm{be}}}_1^\circ
 \leq\frac{\varepsilon_j}{2}(N_c-h)
 \qquad\text{for every }t_0\geq t_*.
\]
Substituting this estimate into \eqref{eq:late-cutoff-basic} yields
$N_{0,2}(F_{t_0},\varepsilon_j)\geq h$ for every $t_0\geq t_*$, which
proves \eqref{eq:late-cutoff}.
Under the assumed semi-strong estimate, if $t_*$ satisfies
\eqref{eq:late-cutoff-time}, then for every $t_0\geq t_*$,
\[
 \frac{2}{\varepsilon_j}
 \norm{F_{t_0}-F_{\mathrm{be}}}_1^\circ
 \leq
 \frac{2C_\lambda}{\varepsilon_j}(1+t_0)^{-\lambda/2}
 \leq
 \frac{2C_\lambda}{\varepsilon_j}(1+t_*)^{-\lambda/2}
 \leq N_c-h.
\]
The conclusion again follows from \eqref{eq:late-cutoff-basic}.
\end{proof}
\begin{proof}[Proof of \cref{thm:conditional-main}]
Set $c_0=\sqrt{NE}$.
\paragraph{1. Finite-time condensation and persistence}
Take $h=N_c/2$ and $\delta=1/4$.  Since $0<N_c\leq N$, these parameters
satisfy the hypotheses of \cref{thm:low-energy-to-condensate}.  By
\eqref{eq:Theta-finite} and \eqref{eq:formation-interval}, we may choose
$j\geq k_0$ so large that $\tau_j(h,\delta)\leq1$.  The assumed
semi-strong convergence and \cref{lem:late-cutoff} provide
$t_*=t_*(h,j)\geq0$ such that
\(
 N_{0,2}(F_{t_*},\varepsilon_j)\geq h.
\)
Applying \cref{thm:low-energy-to-condensate} with $t_0=t_*$ gives a
finite time $T\in[t_*,t_*+\tau_j(h,\delta)]$ such that
\begin{equation}\label{eq:intrinsic-final-atom}
 N_c(T)=F_T(\{0\})
 \geq\frac{3N_c}{8}e^{-c_0(T-t_*)}>0.
\end{equation}
For every $t\geq T$,
\eqref{eq:atom-persistence} yields
\[
 N_c(t)=F_t(\{0\})
 \geq e^{-c_0(t-T)}F_T(\{0\})
 =e^{-\sqrt{NE}(t-T)}N_c(T)>0.
\]
This proves \eqref{eq:atom-positive-main} and
\eqref{eq:atom-persistence-main}.  If $F_0(\{0\})=0$, then
\eqref{eq:intrinsic-final-atom} also implies $T>0$.
\paragraph{2. The asymptotic lower bound for the condensate}
Fix $h\in(0,N_c)$ and $\delta\in(0,1/2)$.  By
\eqref{eq:formation-interval-rate}, every sufficiently large
$j\geq k_0$ satisfies $\tau_j(h,\delta)\leq1$.  For each such $j$, let
$t_*=t_*(h,j)$ be supplied by \cref{lem:late-cutoff}.  If
$t\geq t_*+\tau_j(h,\delta)$, set
$t_0=t-\tau_j(h,\delta)$.  Then $t_0\geq t_*$ and
\(
 N_{0,2}(F_{t_0},\varepsilon_j)\geq h.
\)
Applying \cref{thm:low-energy-to-condensate} on
$[t_0,t_0+\tau_j(h,\delta)]=[t_0,t]$ and evaluating its persistent lower
bound at time $t$ gives
\begin{equation}\label{eq:sliding-lower}
 F_t(\{0\})
 \geq
 (1-\delta)e^{-c_0\tau_j(h,\delta)}h.
\end{equation}
For each sufficiently large fixed $j$, this estimate holds for every
sufficiently large $t$.  Consequently,
\[
 \liminf_{t\to\infty}F_t(\{0\})
 \geq
 (1-\delta)e^{-c_0\tau_j(h,\delta)}h.
\]
Letting $j\to\infty$ and using
$\tau_j(h,\delta)\to0$, we obtain
\[
 \liminf_{t\to\infty}F_t(\{0\})
 \geq(1-\delta)h.
\]
Since this holds for every $h\in(0,N_c)$ and $\delta\in(0,1/2)$,
first letting $\delta\downarrow0$ and then $h\uparrow N_c$ gives
\begin{equation}\label{eq:atom-liminf}
 \liminf_{t\to\infty}F_t(\{0\})\geq N_c.
\end{equation}
\paragraph{3. Strong convergence}
Let $\varphi\in C_b([0,\infty))$ and $a>0$.  Since $N(F_t)=N(F_{\mathrm{be}})$,
\[
 \int\varphi\,d(F_t-F_{\mathrm{be}})
 =
 \int\bigl(\varphi-\varphi(0)\bigr)\,d(F_t-F_{\mathrm{be}}).
\]
Splitting the last integral at $a$ gives
\begin{equation}\label{eq:narrow-from-semistrong}
 \begin{aligned}
 \left|
  \int_{[0,\infty)}
  \varphi\,d(F_t-F_{\mathrm{be}})
 \right|
 \leq
 2N\sup_{0\leq x\leq a}
 |\varphi(x)-\varphi(0)|+
 \frac{2\|\varphi\|_\infty}{a}
 \norm{F_t-F_{\mathrm{be}}}_1^\circ.
 \end{aligned}
\end{equation}
For fixed $a$, the second term tends to zero by the assumed semi-strong
convergence.  Taking the upper limit as $t\to\infty$ and then letting
$a\downarrow0$ shows that $F_t$ converges narrowly to
$F_{\mathrm{be}}$.
Choose a sequence $t_n\to\infty$ such that
\[
 F_{t_n}(\{0\})\longrightarrow
 \limsup_{t\to\infty}F_t(\{0\}).
\]
Since $\{0\}$ is closed, the Portmanteau theorem gives
\[
 \limsup_{t\to\infty}F_t(\{0\})
 \leq F_{\mathrm{be}}(\{0\})=N_c.
\]
Together with \eqref{eq:atom-liminf}, this proves
\[
 N_c(t)=F_t(\{0\})\longrightarrow N_c.
\]
Finally,
\eqref{eq:CL-strong-criterion} gives
\[
 \norm{F_t-F_{\mathrm{be}}}_1
 \leq
 2\left|F_t(\{0\})-N_c\right|
 +C\bigl(
   \norm{F_t-F_{\mathrm{be}}}_1^\circ
  \bigr)^{1/3}
 \longrightarrow0.
\]
This completes the proof.
\end{proof}
\begin{proof}[Proof of \cref{thm:CL-main}]
Fix an arbitrary $\lambda\in(1/20,1/19)$.  By
\cref{prop:CL-relaxation}, there exist a conservative isotropic measure
solution $(F_t)_{t\geq0}$ with initial datum $F_0$ and a constant
$C_\lambda>0$ such that
\[
 \norm{F_t-F_{\mathrm{be}}}_1^\circ
 \leq C_\lambda(1+t)^{-\lambda/2},
 \qquad t\geq0.
\]
Thus \eqref{eq:semistrong-rate-main} holds.  Since $\lambda>0$, this
estimate also gives
\[
 \norm{F_t-F_{\mathrm{be}}}_1^\circ\longrightarrow0
 \qquad\text{as }t\to\infty.
\]
Moreover, $\overline T/\overline T_c<1$ and
\eqref{eq:Nc-positive} give $N_c=F_{\mathrm{be}}(\{0\})>0$.  All the
hypotheses of \cref{thm:conditional-main} are therefore satisfied.  That
theorem provides a finite time $T\geq0$ such that
\[
 F_T(\{0\})>0,
 \qquad
 F_t(\{0\})
 \geq e^{-\sqrt{NE}(t-T)}F_T(\{0\})>0,
 \qquad t\geq T,
\]
and
\[
 F_t(\{0\})\longrightarrow N_c,
 \qquad
 \norm{F_t-F_{\mathrm{be}}}_1\longrightarrow0
 \qquad\text{as }t\to\infty.
\]
These are precisely \eqref{eq:CL-atom-main} and
\eqref{eq:CL-convergence-main}.  Since $\lambda$ was arbitrary, the proof
is complete.
\end{proof}

\section*{Acknowledgments}
S. Luo gratefully acknowledges the hospitality of the Department of
Mathematics at Duke University during his visit in summer 2026, where
part of this work was carried out. 

\section*{Statements and Declarations}
\noindent\textbf{Funding.} No funding was received for conducting this study.

\medskip
\noindent\textbf{Competing interests.}
The authors have no relevant interests to
disclose.

\medskip
\noindent\textbf{Author contributions.}
Jian-Guo Liu proposed the research direction and supervised the project.
Siwei Luo developed the principal new ideas and main arguments, carried
out the mathematical analysis, and wrote the original draft. Both
authors discussed the results, reviewed and edited the manuscript, and
approved the final version.

\medskip
\noindent\textbf{Data availability.} We do not analyse or generate any datasets, because our work proceeds
within a theoretical and mathematical approach.

\medskip
\noindent\textbf{Use of generative AI.} During the preparation of this manuscript, the authors used OpenAI's ChatGPT to assist with literature searches and English-language
editing. All AI-assisted
language was reviewed and revised by the authors, who take full
responsibility for the content and arguments of the
manuscript.
\bibliographystyle{amsalpha}
\bibliography{references}
\end{document}